\documentclass[11pt,a4paper]{article}
\usepackage{amsmath}
\usepackage{amsfonts}
\usepackage{amssymb}
\usepackage{amsthm}
\usepackage{mathrsfs}
\usepackage{dsfont}
\usepackage{paralist}
\usepackage{natbib}
\usepackage{bm}
\usepackage{color}
\usepackage[colorlinks=true,linkcolor=blue,citecolor=blue,pdfborder={0 0 0}]{hyperref}
\usepackage{graphicx}
\usepackage{tabularx}
\usepackage[left=3.5cm,right=3.5cm,top=2cm,bottom=2cm,includeheadfoot]{geometry}
\usepackage{comment}
\usepackage{tikz}
\usetikzlibrary{decorations.pathreplacing}
\usepackage{resizegather}
\DeclareFontFamily{OT1}{pzc}{}
\DeclareFontShape{OT1}{pzc}{m}{it}{<-> s * [1.10] pzcmi7t}{}
\DeclareMathAlphabet{\mathpzc}{OT1}{pzc}{m}{it}

\newcount\Comments  % 0 suppresses notes to selves in text
\newcommand{\kibitz}[2]{\ifnum\Comments=1\textcolor{#1}{#2}\fi}

\newtheoremstyle{normal}% name
{2ex}               % Space above, empty = `usual value'
{3ex}               % Space below
{}                  % Body font
{}                  % Indent amount (empty = no indent, \parindent = para indent)
{\bfseries} % Thm head font
{}                  % Punctuation after thm head
{2pt}   % Space after thm head.
{\thmname{#1}\thmnumber{ #2.} \thmnote{(#3)}}% Thm head spec

\newtheoremstyle{italic}% name
{2ex}%      Space above, empty = `usual value'
{3ex}%      Space below
{\itshape}% Body font
{}%         Indent amount (empty = no indent, \parindent = para indent)
{\bfseries} % Thm head font
{}%        Punctuation after thm head
{2pt}%     Space after thm head.
{\thmname{#1}\thmnumber{ #2.} \thmnote{(#3)}}% Thm head spec

\theoremstyle{normal}
\newtheorem{definition}{Definition}[section]

\newtheorem{example}[definition]{Example}
\newtheorem{condition}[definition]{Condition}

\theoremstyle{italic}
\newtheorem{theorem}[definition]{Theorem}
\newtheorem{lemma}[definition]{Lemma}
\newtheorem{proposition}[definition]{Proposition}

\newcommand\N{\mathbb{N}}

\newcommand\R{\mathbb{R}}

\allowdisplaybreaks[1]

\begin{document}

\title{Testing for simultaneous jumps in case of asynchronous observations}

\author{Ole Martin and  Mathias Vetter\thanks{Christian-Albrechts-Universit\"at zu Kiel, Mathematisches Seminar, Ludewig-Meyn-Str.\ 4, 24118 Kiel, Germany.
{E-mail:} martin@math.uni-kiel.de/vetter@math.uni-kiel.de} \bigskip \\
{Christian-Albrechts-Universit\"at zu Kiel}
}

\maketitle

\begin{abstract}
This paper proposes a novel test for simultaneous jumps in a bivariate It\^o semimartingale when observation times are asynchronous and irregular. Inference is built on a realized correlation coefficient for the jumps of the two processes which is estimated using bivariate power variations of Hayashi-Yoshida type without an additional synchronization step. An associated central limit theorem is shown whose asymptotic distribution is assessed using a bootstrap procedure. Simulations show that the test works remarkably well in comparison with the much simpler case of regular observations.  
\end{abstract}

%\medskip

 high-frequency statistics; It\^o semimartingale; stable convergence

\smallskip


%\tableofcontents

\section{Introduction}
\def\theequation{1.\arabic{equation}}
\setcounter{equation}{0}

Understanding the jump behaviour of a continuous time process is of importance in econometrics, as many decisions in finance are based on knowledge of the path properties of the underlying asset prices. For this reason, a large amount of research over the last decade was concerned with the estimation of certain jump characteristics or with the construction of tests regarding the existence and the nature of the jumps in the respective processes. Quite naturally, the focus was on the univariate setting for most cases, and we refer to the recent monographs \cite{JacPro12} and \cite{AitJac14} as well as to the references cited therein for a overview on statistical methods for (univariate) semimartingales observed in discrete time. 

On the other hand, when it comes to portfolio management and diversification issues there is a clear need for statistical methods which help deciding whether jumps in a specific asset are of idiosyncratic nature or are accompanied by jumps in other assets as well. Starting with \cite{barshe2006}, authors therefore have developed tests for simultaneous jumps in a multivariate framework, but these tests are typically based on the assumption that all components of the multivariate process can be observed synchronously and in a regular fashion. See for example \cite{JacTod09}, \cite{liaand2011} and \cite{mangob2012}. 

A remarkable exception is the test for co-jumps from  \cite{bibwin2015} which is designed for observations including additional noise and works in more general sampling schemes than just regular ones. We will refrain from adding noise in the sequel, but we will keep the focus on irregular observation schemes including asynchronicity in the data. Allowing for such models is much more realistic when it comes to practical applications, as even in the univariate setting observations do not come at equidistant times, and in the case of multivariate processes it is typically the case that not any observation of one component coincides with observations of all the others.  For this reason, there has always been some interest in the generalization of methods for regular sampling schemes to more realistic frameworks. This includes in particular the (simpler) case of continuous It\^o semimartingales. See for example \cite{hayetal2011} or \cite{mykzha2012} for the asymptotic properties of power variations in the univariate setting, or \cite{HayYos05} and \cite{HayYos08} on estimation of covariation for bivariate processes. 

Even more complicated is the situation when the underlying processes contain jumps. In this case, the (few) existing results have mostly focused on the univariate situation. Consistency results for certain power variations can be found in Chapter 3 of \cite{JacPro12}, but associated central limit theorems are only given in the case where jumps do not play a role asymptotically. On the other hand, \cite{BibVet15} provide a central limit theorem which involves non-trivial parts related to jumps, but only in the relatively simple case of realized volatility. 

The aim of the present work therefore is twofold: First, we extend results from \cite{JacTod09}, providing a feasible test for simultaneous jumps of a bivariate process $X =(X^{(1)},X^{(2)})$ over $[0,T],$ when observation times are asynchronous and irregular. As they discriminate between joint and disjoint jumps by estimating an empirical correlation coefficient for the two jump processes, namely 
\begin{align} \label{defPhi}
\Phi_T^{(d)} = \frac{\sum_{s \leq T} \big(\Delta X_s^{(1)} \big)^2 \big(\Delta X_s^{(2)} \big)^2}{\sqrt{\sum_{s \leq T} \big(\Delta X_s^{(1)} \big)^4} \sqrt{\sum_{s \leq T} \big(\Delta X_s^{(2)} \big)^4}}, 
\end{align}
we need an extension of the results from \cite{BibVet15} to a multidimensional framework in order to estimate $\Phi_T^{(d)}$ from irregular sampling schemes as well. Our technique here utilizes the heuristics behind the standard Hayashi-Yoshida estimator for realized covariation in order to identify joint jumps, and we believe that it is of independent interest as quantities such as $\Phi_T^{(d)}$ also play a central role in various other situations related to inference on jump processes. 

Second, under the null hypothesis of no joint jumps we provide an associated central limit theorem for our estimator of $\Phi_T^{(d)}$. As the limiting variable not only depends in a complicated way on the characteristics of $X$, but also on unknown variables which are due to the fine structure of the sampling scheme, we provide a bootstrap procedure in order to estimate critical values of our final test statistic. An extensive simulation study shows that our test has a similar finite sample behaviour as the standard test by \cite{JacTod09} when the (random) number of observations in both components equals on average the fixed number of observations in the simple regular case. This is remarkable when it comes to practical applications, as no additional synchronization step is necessary which inevitably causes a loss of data and therefore leads to a loss in effeciency. 

The remainder of the paper is organized as follows: Section \ref{sec:setting} deals with the formal setting in this work, and we introduce our estimator for $\Phi_T^{(d)}$ as well as minor assumptions under which consistency holds. In Section \ref{sec:clt} we need stronger conditions, as we are interested in the associated central limit theorem. The bootstrap procedure leading to the final test statistic is introduced in Section \ref{sec:test}, while its finite sample properties are investigated in Section \ref{sec:simul}. All proofs are gathered in the Appendix, which is Section \ref{sec:proof}.

\section{Setting and test statistic} \label{sec:setting}
\def\theequation{2.\arabic{equation}}
\setcounter{equation}{0}
Our goal in the sequel is to derive a statistical test based on high-frequency observations which allows to decide whether two processes do jump at a common time or not. We consider the following model for the process and the observation times: Let $X =(X^{(1)},X^{(2)})^*$ be a two-dimensional It\^o semimartingale on $(\Omega,\mathcal{F},\mathbb{P})$ of the form
\begin{multline}
\label{ItoSemimart}
X_t = X_0 + \int \limits_0^t b_s ds + \int \limits_0^t \sigma_s dW_s + \int \limits_0^t \int \limits_{\R^2} \delta(s,z)\mathds{1}_{\{\|\delta(s,z)\|\leq 1\}} (\mu - \nu)(ds,dz) \\
+ \int \limits_0^t \int \limits_{\R^2}  \delta(s,z) \mathds{1}_{\{\|\delta(s,z)\|> 1\}} \mu(ds,dz),
\end{multline}
where $W=(W^{(1)},W^{(2)})^*$ is a two-dimensional Brownian motion with covariation $d[W^{(1)},W^{(2)}]_t = \rho_t dt$, $\mu$ is a Poisson random measure on $\mathbb{R}^+ \times\mathbb{R}^2$, and its predictable compensator satisfies \mbox{$\nu(ds,dz)=ds \otimes \lambda(dz)$} for some $\sigma$-finite measure $\lambda$ on $\mathbb{R}^2$ endowed with the Borelian $\sigma$-algebra. $b$ is a two-dimensional adapted process,  $\sigma=\text{diag}(\sigma^{(1)}, \sigma^{(2)})$ is a $(2 \times 2)$-dimensional process and $\delta$ is a two-dimensional predictable process on $\Omega \times \mathbb{R}^+ \times \mathbb{R}^2$. $\sigma_s^{(1)}$, $\sigma_s^{(2)}$ and $\rho_s$ are all univariate adapted. We write $\Delta X_s=X_s-X_{s-}$ with $X_{s-}=\lim_{t \nearrow s} X_t$ for a possible jump of $X$ in $s$. 

The observation times are given by 
\[
\pi_n =\big\{\big(t_{i,n}^{(1)} \big)_{i \in \mathbb{N}_0},\big(t_{i,n}^{(2)} \big)_{i \in \mathbb{N}_0}  \big\}, \quad n \in \N,
\]
where $\big(t_{i,n}^{(l)}\big)_{i \in \mathbb{N}_0},~ l=1,2,$ are increasing sequences of stopping times with $t_{0,n}^{(l)}=0$. 
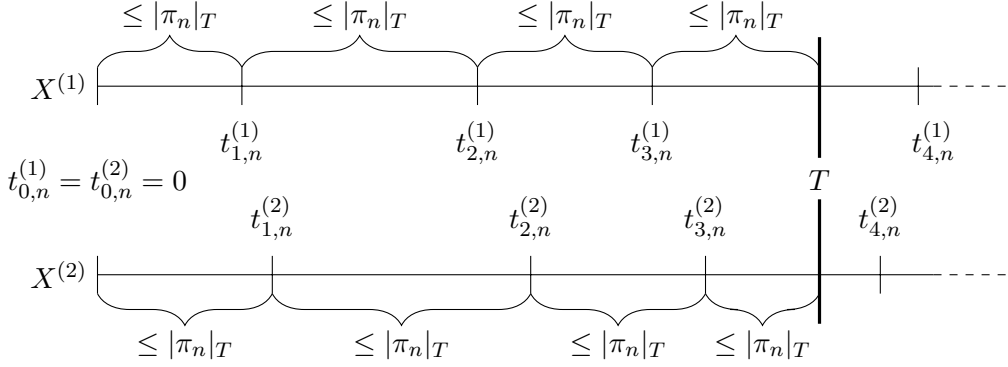
\begin{figure}[t]
\centering
\hspace{-0.6cm}
\begin{tikzpicture}
\draw (0,1.75) -- (11,1.75)
      (0,-0.75) -- (11,-0.75)
      (0,1.5) -- (0,2)
		(1.9,1.5) -- (1.9,2)
		(5,1.5) -- (5,2)
		(7.3,1.5) -- (7.3,2)
		
		(10.8,1.5) -- (10.8,2)
		(0,-0.5) -- (0,-1)
		(2.3,-0.5) -- (2.3,-1)
		(5.7,-0.5) -- (5.7,-1)
		(8,-0.5) -- (8,-1)
		
		(10.3,-0.5) -- (10.3,-1);
\draw[dashed] (11,1.75) -- (12,1.75)
      (11,-0.75) -- (12,-0.75);
\draw[very thick] (9.5,-1.4) -- (9.5,0.25)
      (9.5,0.8) -- (9.5,2.4);
\draw	(0,0.5) node{$t_{0,n}^{(1)}=t_{0,n}^{(2)}=0$}
		(1.9,1) node{$t_{1,n}^{(1)}$}
		(5,1) node{$t_{2,n}^{(1)}$}
		(7.3,1) node{$t_{3,n}^{(1)}$}
		(11,1) node{$t_{4,n}^{(1)}$}
		(9.5,0.5) node{\textbf{$T$}}
		(2.3,0) node{$t_{1,n}^{(2)}$}
		(5.7,0) node{$t_{2,n}^{(2)}$}
		(8,0) node{$t_{3,n}^{(2)}$}
		(10.3,0) node{$t_{4,n}^{(2)}$};
\draw   (0,1.75) node[left,xshift=-0pt]{$X^{(1)}$}
(0,-0.75) node[left,xshift=-0pt]{$X^{(2)}$};
\draw[decorate,decoration={brace,amplitude=12pt}]
	(0,2)--(1.9,2) node[midway, above,yshift=10pt,]{$\leq|\pi_n|_T$};\draw[decorate,decoration={brace,amplitude=12pt}]
	(1.9,2)--(5,2) node[midway, above,yshift=10pt,]{$\leq|\pi_n|_T$};\draw[decorate,decoration={brace,amplitude=12pt}]
	(5,2)--(7.3,2) node[midway, above,yshift=10pt,]{$\leq|\pi_n|_T$};\draw[decorate,decoration={brace,amplitude=12pt}]
	(7.3,2)--(9.5,2) node[midway, above,yshift=10pt,]{$\leq|\pi_n|_T$};	
\draw[decorate,decoration={brace,amplitude=12pt}]
	(9.5,-1)--(8,-1) node[midway, below,yshift=-10pt,]{$\leq|\pi_n|_T$};
	\draw[decorate,decoration={brace,amplitude=12pt}]
	(8,-1)--(5.7,-1) node[midway, below,yshift=-10pt,]{$\leq|\pi_n|_T$};
	\draw[decorate,decoration={brace,amplitude=12pt}]
	(5.7,-1)--(2.3,-1) node[midway, below,yshift=-10pt,]{$\leq|\pi_n|_T$};
	\draw[decorate,decoration={brace,amplitude=12pt}]
	(2.3,-1)--(0,-1) node[midway, below,yshift=-10pt,]{$\leq|\pi_n|_T$};
\end{tikzpicture}
\caption{A realization of the observation scheme $\pi_n$ restricted to $[0,T]$.}
\end{figure}
By
\[
|\pi_n|_T=\sup \big\{t_{i,n}^{(l)} \wedge T -t_{i-1,n}^{(l)} \wedge T \big|i\geq 1,~ l=1,2   \big\}
\]
we denote the mesh of the observation times up to $T$. Throughout the paper we use $n$ as an unobservable variable governing the observations and the asymptotics which does not appear in the statistics used later on. 

We introduce the following subsets of $\Omega$ to formalize the hypotheses:
\begin{equation*}
\begin{aligned}
    &\Omega_T^{(d)}= \{\omega \in \Omega :\exists s_1,s_2 \in [0,T] \text{ with } \Delta X_{s_1}^{(1)} \neq 0 \text{ and } \Delta X_{s_2}^{(2)} \neq 0, \\
    &~~~~~~~~~~~~~~~~~~~~~~~~~~~~~~~~~~~~~~~~~~~\text{but } \Delta X_s^{(1)}\Delta X_s^{(2)}=0 ~\forall s \in [0,T]     \}, \\
    &\Omega_T^{(j)}= \{\omega \in \Omega :\exists s \in [0,T] \text{ with } \Delta X_s^{(1)}\Delta X_s^{(2)}\neq 0   \}, \\
    &\Omega_T^{(c)}= \{\omega \in \Omega :\Delta X_{s}^{(1)}=0 ~\forall s \in [0,T]  \text{ or }  \Delta X_{s}^{(2)}=0 ~\forall s \in [0,T]\}.
\end{aligned}
\end{equation*}
Hence $\Omega_T^{(d)}$ is the set where $X^{(1)}$ and $X^{(2)}$ are both discontinuous on $[0,T]$ but do not jump together, $\Omega_T^{(j)}$ is the set where $X^{(1)}$ and $X^{(2)}$ have at least one common jump in $[0,T]$, and $\Omega_T^{(c)}$ is the set where at least one of the processes $X^{(1)}$ or $X^{(2)}$ is continuous on $[0,T]$. Our goal in this paper is to find a testing procedure for deciding whether an observation is from $\Omega_T^{(d)}$ or from $\Omega_T^{(j)}$. This means in particular that we focus on a specific path of $X$, and it might be the case that the underlying model allows for joint jumps but none of them occurs on the observed path up to time $T$. In such a case, the hypothesis of joint jumps should be rejected. Also, it is reasonable to apply a test for jumps in any of the processes (like the one from \cite{aitjac2009}) prior to the analysis, as one does not know a priori whether $\omega \in \Omega_T^{(c)}$ or not.  

All our test statistics are based on the increments
\[
\Delta_{i,n}^{(l)} X= X_{t_{i,n}^{(l)}}-X_{t_{i-1,n}^{(l)}}, \quad i \geq 1, \quad l=1,2,
\]
and we denote by $\mathcal{I}_{i,n}^{(l)}=\big(t_{i-1,n}^{(l)},t_{i,n}^{(l)}\big],~l=1,2,$ the corresponding observation intervals. For a function $f:\R^2 \rightarrow \R$ we set
\[
V(f,\pi_n)_T=\sum_{i,j:t_{i,n}^{(1)} \wedge t_{j,n}^{(2)}\leq T} f\big(\Delta_{i,n}^{(1)} X^{(1)},\Delta_{j,n}^{(2)} X^{(2)} \big)\mathds{1}_{\{\mathcal{I}_{i,n}^{(1)} \cap \mathcal{I}_{j,n}^{(2)}\neq \emptyset \}}
\]
in the style of the Hayashi-Yoshida estimator for the quadratic covariation (\cite{HayYos05}), and for a function $g: \R \rightarrow \R$ we define
\[
V^{(l)}(g,\pi_n)_T=\sum_{i:t_{i,n}^{(l)}\leq T} g\big(\Delta_{i,n}^{(l)} X^{(l)}\big), \quad l=1,2.
\]
In particular, as we are interested in estimating $\Phi_T^{(d)}$ from (\ref{defPhi}), we consider these expressions for the functions $f(x)=(x_1 x_2)^2$ and $g(x)=x^4$. Then our main statistic becomes
\[
\widetilde{\Phi}_{n,T}^{(d)}=\frac{V(f,\pi_n)_T}{\sqrt{V^{(1)}(g,\pi_n)_T V^{(2)}(g,\pi_n)_T}},
\]
whose asymptotics we are going to study and which will be used to construct an asymptotic test. 

In order to describe the asymptotics of $\widetilde{\Phi}_{n,T}^{(d)}$ we set
\begin{equation*}
\begin{aligned}
B_T=\sum_{s \leq T} \big(\Delta X_s^{(1)} \big)^2 \big(\Delta X_s^{(2)} \big)^2, \quad B_T^{(l)}=\sum_{s \leq T} \big(\Delta X_s^{(l)} \big)^4 \text{~for~} l=1,2,
\end{aligned}
\end{equation*}
so that 
\[
\Phi_T^{(d)}=\frac{B_T}{\sqrt{B_T^{(1)}B_T^{(2)}}}.
\]
Obviously, $\Phi_T^{(d)}$ is well-defined on the complement of $\Omega_T^{(c)}$ only, and in this case it can be interpreted as the correlation between the squared jumps of $X^{(1)}$ and $X^{(2)}$: $\Phi_T^{(d)}$ is always in $[0,1]$, and it is equal to $0$ if and only if there  are no common jumps and equal to $1$ if and only if there exists a constant $c > 0$ with $\big(\Delta X_s^{(1)} \big)^2=c\big(\Delta X_s^{(2)} \big)^2$ for all $s \leq T$. 

In order to derive results on the asymptotic behaviour of $\widetilde{\Phi}_{n,T}^{(d)}$, we require the following restrictions on the process $X$ and the observation scheme $\pi_n$.

\begin{condition}\label{cond_consistency}
The processes $b_s,\sigma_s^{(1)},\sigma_s^{(2)},\rho_s$ and $s \mapsto \delta(s,z)$ are continuous on $[0,T]$. Furthermore, we have $\|\delta(s,z)\| \leq \gamma(z)$ for some bounded function $\gamma$ which satisfies $\int (1 \wedge \gamma^2(z))\lambda(dz) < \infty$. The sequence of observation schemes $(\pi_n)_n$ fulfills
\[
|\pi_n|_T \overset{\mathbb{P}}{\longrightarrow} 0.
\]
\end{condition}

The conditions on the components of $X$ are not very restrictive and might even be further relaxed as in \cite{JacTod09}. We impose stronger restrictions here in order to keep the notation and the proofs simpler. The condition that the mesh vanishes is a minimal condition on the observation scheme, since we consider properties like the presence of jumps in the observed path which depend on the complete path in continuous time. 

Regarding the observation scheme, we are able to work in the general setting of increasing stopping times with vanishing mesh in order to derive consistency of the estimator $\widetilde{\Phi}_{n,T}^{(d)}$. This result might be of its own interest, as it generalizes results from Section 3 of \cite{JacPro12} to the case of asynchronicity. However, for the construction of a central limit theorem in Section \ref{sec:clt} we are not able to work within this general setting. Although in practice a theory for endogeneous observation times might be desirable, previous research shows that even in simple situations it is difficult to derive central limit theorems (see \cite{FukRos12} or \cite{VetZwi16}). For this reason we restrict ourselves in Section \ref{sec:clt} to exogeneous observation times which still cover a lot of random and irregular sampling schemes. We will see that already in this setting the asymptotic theory becomes significantly more difficult compared to the framework of equidistant observations. 

Speaking of consistency only, we are able to prove  
\begin{align}
&V(f,\pi_n)_T \overset{\mathbb{P}}{\longrightarrow}B_T, \label{cons_Vf}
\\&V^{(l)}(g,\pi_n)_T \overset{\mathbb{P}}{\longrightarrow}B_T^{(l)}, \quad l=1,2,\label{cons_Vg}
\end{align}
whenever Condition \ref{cond_consistency} holds. Note that \eqref{cons_Vg} already follows from Theorem 3.3.1 in \cite{JacPro12} while the first statement \eqref{cons_Vf} needs a generalization of this theorem to the setting of asynchronous observations. 

\begin{theorem}\label{cons_theo}
Let $X$ be an It\^o semimartingale of the form \eqref{ItoSemimart} and $(\pi_n)_n$ be a sequence of observation schemes such that Condition \ref{cond_consistency} is fulfilled. Then we have
\begin{align*}
\widetilde{\Phi}_{n,T}^{(d)} \overset{\mathbb{P}}{\longrightarrow} \Phi_T^{(d)}
\end{align*}
on the complement of $\Omega_T^{(c)}$.
\end{theorem}

Theorem \ref{cons_theo} states that $\widetilde{\Phi}_{n,T}^{(d)}$ converges to $0$ on the set $\Omega_T^{(d)}$ and to a strictly positive limit on $\Omega_T^{(j)}$. So a natural test for the null $\omega \in \Omega_T^{(d)}$ against $\omega \in \Omega_T^{(j)}$ makes use of a critical region of the form
\begin{align}\label{critical_region}
\mathcal{C}_n=\big\{\widetilde{\Phi}_{n,T}^{(d)}>\mathpzc{c}_n\big\}
\end{align}
for a suitable, possibly random sequence $(\mathpzc{c}_n)_{n \in \mathbb{N}}$. In order to choose $\mathpzc{c}_n$ such that the test has a certain level $\alpha$ we need knowledge of the asymptotic behaviour of $\widetilde{\Phi}_{n,T}^{(d)}$ on $\Omega_T^{(d)}$, which will be developed in form of a central limit theorem in the next section.

\section{Central limit theorem} \label{sec:clt}
\def\theequation{3.\arabic{equation}}
\setcounter{equation}{0}

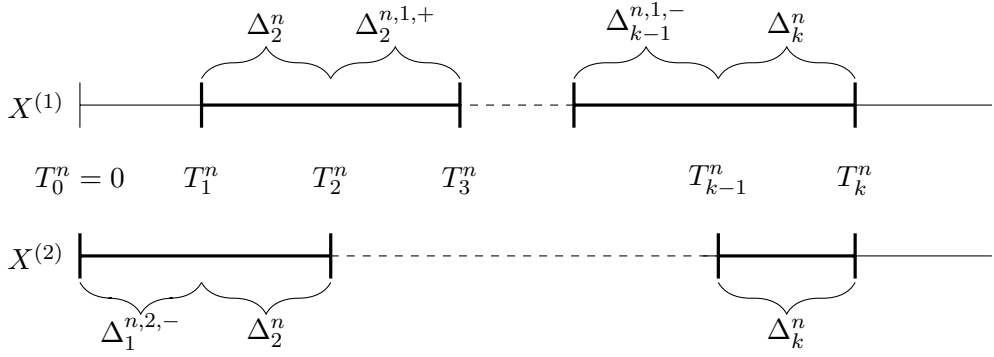
\begin{figure}[t]
\centering
\hspace{-0.6cm}
\begin{tikzpicture}
\draw[dashed] (5,1) -- (8,1)
			(3.8,-1) -- (8.9,-1);
\draw[very thick] (2.1,1) -- (5.5,1)
		(7,1) -- (10.7,1)
		(2.1,0.7) -- (2.1,1.3)
		(5.5,0.7) -- (5.5,1.3)
		(7,0.7) -- (7,1.3)
		(10.7,0.7) -- (10.7,1.3)
		(0.5,-1) -- (3.8,-1)
		(8.9,-1) -- (10.7,-1)
		(0.5,-0.7) -- (0.5,-1.3)
		(3.8,-0.7) -- (3.8,-1.3)
		(8.9,-0.7) -- (8.9,-1.3)
		(10.7,-0.7) -- (10.7,-1.3);
\draw (0.5,1) -- (5,1)
	  (8,1) -- (12.5,1)
		(0.5,-1) -- (3.7,-1)
		(8.7,-1) -- (12.5,-1)
		(0.5,0.7) -- (0.5,1.3);
\draw (0.5,0) node{$T_0^n=0$}
		(2.1,0) node{$T_1^n$}
		(3.8,0) node{$T_2^n$}
		(5.5,0) node{$T_3^n$}
		(8.9,0) node{$T_{k-1}^n$}
		(10.7,0) node{$T_{k}^n$};	
\draw[decorate,decoration={brace,amplitude=12pt}]
	(2.1,1.35)--(3.8,1.35) node[midway, above,yshift=10pt,]{$\Delta_2^n$};
	\draw[decorate,decoration={brace,amplitude=12pt}]
	(3.8,1.35)--(5.5,1.35) node[midway, above,yshift=10pt,]{$\Delta_2^{n,1,+}$};
	\draw[decorate,decoration={brace,amplitude=12pt}]
	(7,1.35)--(8.9,1.35) node[midway, above,yshift=10pt,]{$\Delta_{k-1}^{n,1,-}$};
	\draw[decorate,decoration={brace,amplitude=12pt}]
	(8.9,1.35)--(10.7,1.35) node[midway, above,yshift=10pt,]{$\Delta_k^{n}$};
\draw[decorate,decoration={brace,amplitude=12pt}]
	(10.7,-1.35)--(8.9,-1.35) node[midway, below,yshift=-10pt,]{$\Delta_k^{n}$};
	\draw[decorate,decoration={brace,amplitude=12pt}]
	(3.8,-1.35)--(2.1,-1.35) node[midway, below,yshift=-10pt,]{$\Delta_2^{n}$};
	\draw[decorate,decoration={brace,amplitude=12pt}]
	(2.1,-1.35)--(0.5,-1.35) node[midway, below,yshift=-8pt,]{$\Delta_1^{n,2,-}$};
\draw (0.5,1) node[left,xshift=-2pt]{$X^{(1)}$}
(0.5,-1) node[left,xshift=-2pt]{$X^{(2)}$};
\end{tikzpicture}
\caption{Merged observation times and interval lengths to previous and upcoming observation times.}
\end{figure}

In order to derive a central limit theorem we first have to specify the asymptotics of the observation scheme. We start by defining several quantities which depend on the stopping times only. Following \cite{BibVet15}, we merge the observation times of $X^{(1)}$ and $X^{(2)}$ into a single observation scheme given by
\begin{equation*}
\begin{aligned}
T^n_0&=0,
\\T^n_{k}&=\inf \{t_{i,n}^{(l)}| t_{i,n}^{(l)}>T^n_{k-1}, l=1,2\}, \quad k \geq 1.
\end{aligned}
\end{equation*}
We set $\Delta_k^n=T^n_k-T^n_{k-1}$ and
\begin{align*}
\Delta^{n,l,-}_k=T^n_k-\sup \{ t_{i,n}^{(l)}| t_{i,n}^{(l)} \leq T^n_k \}, \quad \Delta^{n,l,+}_k=\inf \{ t_{i,n}^{(l)}|t_{i,n}^{(l)} \geq T^n_k \}-T^n_k
\end{align*}
for $l=1,2$.

By 
\begin{align*}
\tau_{n,+}^{(l)}(s)=\inf\{t_{i,n}^{(l)}|t_{i,n}^{(l)} \geq s\}, \quad
\tau_{n,-}^{(l)}(s)=\sup\{t_{i,n}^{(l)}|t_{i,n}^{(l)} \leq s\}, \quad l=1,2,
\end{align*}
we denote the observation times immediately before and after time $s$. Using this notation we set
\begin{equation*}
\begin{aligned}
\mathcal{M}^{(l)}_n (s)=\tau_{n,+}^{(l)}(\tau_{n,+}^{(3-l)}(s))-\tau_{n,-}^{(l)}(\tau_{n,-}^{(3-l)}(s)), \quad l=1,2,
\end{aligned}
\end{equation*}
for the total length of the observation intervals of process $X^{(l)}$ which overlap with the observation interval of $X^{(3-l)}$ containing $s$. Let $i_n^{(l)}(s)$ denote the index of the observation interval of $X^{(l)}$ containing $s$, i.e. $i_n^{(l)}(s)$ is defined via $$s \in \mathcal{I}_{i_n^{(l)}(s),n}^{(l)}.$$ Over the intervals which make up $\mathcal{M}^{(l)}_n (s)$ we denote with $\eta^{(l)}_n(s)$, given by
\begin{equation}\label{def_eta_m}
\begin{aligned}
&\eta^{(1)}_n(s)=\sum_{j:\mathcal{I}_{j,n}^{(1)} \leq T} \big(\Delta_{i,n}^{(1)} W^{(1)}\big)^2 \mathds{1}_{\big\{\mathcal{I}_{j,n}^{(1)} \cap \mathcal{I}_{i^{(2)}_n(s),n}^{(2)} \neq \emptyset\big\}},
\\&\eta^{(2)}_n(s)=\sum_{j:\mathcal{I}_{j,n}^{(2)} \leq T} \big(\Delta_{j,n}^{(2)} W^{(2)}\big)^2
\mathds{1}_{\big\{\mathcal{I}_{i^{(1)}_n(s),n}^{(1)} \cap \mathcal{I}_{j,n}^{(2)} \neq \emptyset\big\}},
\end{aligned}
\end{equation}
the sum over the squared increments of the respective Brownian motions driving the processes $X^{(l)}$ and $X^{(2)}$. See Figure \ref{illu} for an illustration. Even though the driving Browian motions are in general dependent,  we will see that the limiting variables of $\eta^{(1)}_n(s)$ and $\eta^{(2)}_n(s)$ can be chosen to be independent, as under the null hypothesis both variables never occur at the same time in the limit.

\begin{figure}[t] \label{illu}
\centering
\hspace{-0.6cm}
\begin{tikzpicture}
\draw[dashed] (6,1.25) -- (6,-1.25) node[midway, xshift=10pt]{$s$}
			(1,-1) -- (3,-1)
			(7,-1) -- (13,-1)
			(1,1) -- (5,1)
			(11,1)--(13,1);
\draw[very thick] (5,1)--(11,1)
			(3,-1)--(7,-1) 
			(5,0.7)--(5,1.3)
			(11,0.7)--(11,1.3)
			(3,-0.7)--(3,-1.3)
			(7,-0.7)--(7,-1.3);
\draw (2,0.7) -- (2,1.3)
		(8,-1.3) -- (8,-0.7)
		(9.1,-1.3) -- (9.1,-0.7)
		(11.7,-1.3) -- (11.7,-0.7)
;
\draw[decorate,decoration={brace,amplitude=12pt}]
	(2,1.35)--(11,1.35) node[midway, above,yshift=10pt,]{$\mathcal{M}_n^{(1)}(s)$};
	
\draw[decorate,decoration={brace,amplitude=12pt}]
	(11.7,-1.35)--(3,-1.35) node[midway, below,yshift=-10pt,]{$\mathcal{M}_n^{(2)}(s)$};	
	\draw (1,1) node[left,xshift=-2pt]{$X^{(1)}$}
(1,-1) node[left,xshift=-2pt]{$X^{(2)}$};
\end{tikzpicture}
\caption{Intervals around time $s$.}
\end{figure}
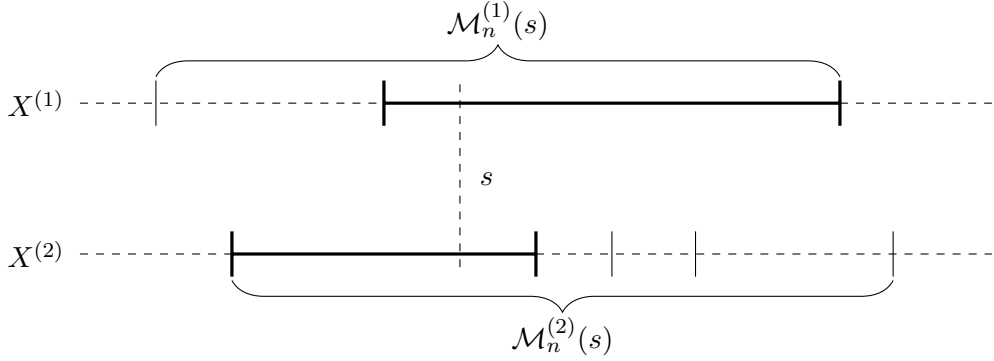

The following condition comprises the assumptions on the asymptotics of the sequence of observation schemes $(\pi_n)_n$ which are needed for the derivation of a central limit theorem. While the first one is a rather mild assumption on the mesh of the sampling scheme, the other two conditions ensure a kind of local regularity which is needed to deduce convergence both of the purely continuous part and the cross part in the limit. 

\begin{condition} \label{cond_centr}
The process $X$ and the sequence of observation schemes $(\pi_n)_n$ fulfill Condition \ref{cond_consistency}, and the observation times are exogeneous, i.e.\ independent of the process $X$ and its components.
\begin{enumerate}[(i)]
\item
It holds
\begin{align*}
\mathbb{E}[(|\pi_n|_T)^{3/2}]=o(n^{-1}).
\end{align*}
\item
The functions
\begin{equation*}
\begin{aligned}
G_n(t)&=n\sum_{k:T_k^n  \leq t} \big(\Delta_k^n \big)^2 ,
\\H_n(t)&=n\sum_{k:T_k^n  \leq t} \big(\Delta_{k-1}^{n,1,-}+\Delta_k^n+\Delta_{k}^{n,1,+} \big)\big(\Delta_{k-1}^{n,2,-}+\Delta_k^n+\Delta_{k}^{n,2,+} \big),
\end{aligned}
\end{equation*}
converge pointwise on $[0,T]$ in probability to continuously differentiable functions $G,H$.
\item
The integral
\begin{multline*}
\int_{[0,T]^{k_1+k_2}} g(x_1,\dots,x_{k_1},x_1',\dots,x_{k_2}') \mathbb{E} \Bigg[ \prod_{p=1}^{k_1} h^{(1)}_p \big(n\eta_n^{(1)}(x_p)\big)
\\\times \prod_{p=1}^{k_2} h^{(2)}_p \big(n\eta_n^{(2)}(x_p')\big)\Bigg] dx_{k_1} \dots dx_1dx_{k_2}' \dots dx_1'
\end{multline*}
converges for $n \rightarrow \infty$ to
\begin{multline}\label{cond_centr_iii2}
\int_{[0,T]^{k_1+k_2}} g(x_1,\dots,x_{k_1},x_1',\dots,x_{k_2}')\prod_{p=1}^{k_1} \int_{\mathbb{R}} h^{(1)}_p \big(y\big)\Gamma^{(1)}(x_p,dy)
\\\times \prod_{p=1}^{k_2} \int_{\mathbb{R}} h^{(2)}_p \big(y'\big)\Gamma^{(2)}(x_p',dy')
 dx_{k_1} \dots dx_1dx_{k_2}' \dots dx_1'
\end{multline}
for all bounded continuous functions $g:\mathbb{R}^{k_1+k_2} \rightarrow \mathbb{R}$ and $h^{(l)}_p:\mathbb{R}\rightarrow \mathbb{R},~l=1,2$. Here $\Gamma^{(l)}(\cdot,dy),~l=1,2,$ are families of probability measures on $[0, T]$ which admit densities such that the first moments are uniformly bounded.
\end{enumerate}
\end{condition}

Because of the exogeneity of the observation times we may assume in the following that the probability space has the form
\begin{align*}
(\Omega,\mathcal{F},\mathbb{P})=(\Omega_\mathcal{X} \times \Omega_\mathcal{S},\mathcal{X} \otimes \mathcal{S},\mathbb{P}_\mathcal{X} \otimes \mathbb{P}_\mathcal{S}),
\end{align*}
where $\mathcal{X}$ denotes the $\sigma$-algebra generated by $X$ and its components and $\mathcal{S}$ denotes the $\sigma$-algebra generated by the observation schemes $(\pi_n)_n$. We will from here on consider $\Omega_T^{(c)},\Omega_T^{(d)},\Omega_T^{(j)}$ to be subsets of $\Omega_\mathcal{X}$.

As usual when power variations for orders higher than two are considered, the limiting term in the central limit theorem will be comprised of a continuous term and a cross term which contains the continuous part of one process and the jumps of the other process. The term originating from the continuous terms is given by
\begin{align*}
\widetilde {C}_T=\int_0^T \big(2(\rho_s\sigma_s^{(1)}\sigma_s^{(2)})^2G'(s) +(\sigma_s^{(1)}\sigma_s^{(2)})^2 H'(s) \big)ds.
\end{align*}
Here, the functions $G'$ and $H'$ are a measure for the asymptotic density of observation times in a given time interval. Two different functions are needed because the products of increments over overlapping and non-overlapping observation intervals have different variances.

The limiting term originating from the cross terms of continuous parts and jumps is given by
\begin{align*}
\widetilde{D}_T=\sum_{p:S_p \leq T} \big( \big(\Delta X_{S_p}^{(1)}\big)^2\big(\sigma^{(2)}_{S_p}\big)^2 \eta^{(2)}(S_p)
+\big(\Delta X_{S_p}^{(2)}\big)^2\big(\sigma^{(1)}_{S_p}\big)^2 \eta^{(1)}(S_p)\big),
\end{align*}
where $(S_p)_{p \geq 0}$ is an enumeration of the jump times of $X$ and the $\eta^{(l)}(S_p)$ are random variables defined on an extended probability space $(\widetilde{\Omega},\widetilde{\mathcal{A}},\widetilde{\mathbb{P}})$. Their distribution is given by
\[
\eta^{(l)}(x) \sim \Gamma^{(l)}(x,dy),
\]
where the $\eta^{(l)}(x)$ are independent of each other and independent of the process $X$ and its components. It is worth mentioning that we do not consider common jumps, since we derive the central limit theorem under the null hypothesis of no common jumps. This leads to independent $\eta^{(l)}(x)$ which simplifies the structure of the limiting variables compared to \cite{BibVet15}. 

Using the above notation we derive the following central limit theorem on $\Omega_T^{(d)}$. 
\begin{theorem}\label{clt}
If Condition \ref{cond_centr} is fulfilled, we have the $\mathcal{X}$-stable convergence 
\begin{align}\label{clt_equation}
n \widetilde{\Phi}_{n,T}^{(d)} \overset{\mathcal{L}-s}{\longrightarrow} \widetilde{\Psi}_T=\frac{\widetilde {C}_T+\widetilde {D}_T}{\sqrt{B_T^1 B_T^2}}
\end{align}
on the set $\Omega_T^{(d)}$.
\end{theorem}

The central limit theorem states that $n \widetilde{\Phi}_{n,T}^{(d)}$ converges $\mathcal{X}$-stably in law on the set $\Omega_T^{(d)}$ to a random variable $\widetilde{\Psi}_T$ on an extended probability space $(\widetilde{\Omega},\widetilde{\mathcal{A}},\widetilde{\mathbb{P}})$ which means that we have
\[
\mathbb{E} \big[g\big(n \widetilde{\Phi}_{n,T}^{(d)}  \big) Y\mathds{1}_{\Omega_T^{(d)}}\big] \rightarrow \widetilde{\mathbb{E}}\big[g\big(\widetilde{\Psi}_T \big) Y\mathds{1}_{\Omega_T^{(d)}}\big]
\] 
for all bounded and continuous functions $g$ and all $\mathcal{X}$-measurable bounded random variables $Y$. For more background information on stable convergence in law we refer to \cite{JacPro12}, \cite{JacShi02} and \cite{PodVet10}.

\begin{example}\label{ext_equid}
Let us discuss the standard setting of equidistant and synchronous observations times. In this case, $t_{i,n}^{(l)}=i/n$, so we have $|\pi_n|_T=n^{-1}$. Hence Condition \ref{cond_consistency} and Condition \ref{cond_centr}(i) are trivially fulfilled. Furthermore,
\begin{align*}
H_n(t)=G_n(t)=n\sum_{i=1}^{\lfloor t/n \rfloor} \big(1/n \big)^2 \rightarrow  t,
\end{align*}
which yields Condition \ref{cond_centr}(ii). We also have $n \eta_n^{(l)}(s) \sim \chi_1^2$, so the limiting distribution $\eta^{(l)}(s)$ is known to be $\chi_1^2$. Standard arguments finally show that $n \eta^{(l)}_n(S_p)$ and $n \eta^{(l')}_n(S_{p'})$ are asymptotically independent for different jump times $S_p \neq S_{p'}$. Hence Condition \ref{cond_centr}(iii) is satisfied and we have \eqref{clt_equation} with
\begin{align*}
\widetilde {C}_T&=\int_0^T \big(2(\rho_s\sigma_s^{(1)}\sigma_s^{(2)})^2 +(\sigma_s^{(1)}\sigma_s^{(2)})^2  \big)ds,
\\
\widetilde{D}_T&=\sum_{p:S_p \leq T} \big( \big(\Delta X_{S_p}^{(1)}\big)^2\big(\sigma^{(2)}_{S_p}\big)^2 \big(U^{(1)}_p\big)^2
+\big(\Delta X_{S_p}^{(2)}\big)^2\big(\sigma^{(1)}_{S_p}\big)^2 \big(U^{(2)}_p\big)^2\big),
\end{align*}
for independent standard normal distributed random variables $U^{(l)}_p$, $l=1,2$. Of course, these terms are identical to the corresponding terms $C_T$ and $\widetilde{D}_T$ in (3.12) and (3.14) of \cite{JacTod09}, and Theorem \ref{clt} becomes Theorem 4.1(a) of \cite{JacTod09} in this setting.
\end{example}

In order to illustrate the theory laid out above we also want to discuss a truly irregular and random setting. Specifically, we consider observation times which are given by the jump times of Poisson processes, but our conditions cover various other sampling schemes as well. Note that Poisson sampling has been discussed frequently in the literature; see e.g. \cite{BibVet15} and \cite{HayYos08}.

\begin{example}\label{poiss1}
Let the observation times of $X^{(1)}$ and $X^{(2)}$ be given by the jump times of independent Poisson processes with intensities $n\lambda_1$ and $n \lambda_2$. Lemma 8 from \cite{HayYos08} states 
\begin{align}\label{Poisson_mesh}
\mathbb{E}\big[ (|\pi_n|_T)^q \big]=o(n^{-\alpha})
\end{align} 
for any $0 \le \alpha < q$, so both Condition \ref{cond_consistency} and Condition \ref{cond_centr}(i) are satisfied. In addition, \mbox{Proposition 1} in \cite{HayYos08} gives Condition \ref{cond_centr}(ii) via
\begin{small}
\begin{equation*}
\begin{aligned}
G_n(t)&=n\sum_{i,j:t_{i,n}^{(1)} \wedge t_{j,n}^{(2)} \leq t} \big|\mathcal{I}_{i,n}^{(1)} \cap \mathcal{I}_{j,n}^{(2)}\big|^2 \mathds{1}_{\{\mathcal{I}_{i,n}^{(1)} \cap \mathcal{I}_{j,n}^{(2)}\neq \emptyset\}}
+nO((|\pi_n|_t)^2)
\overset{\mathbb{P}}{\longrightarrow} \frac{2}{\lambda_1+\lambda_2}t,
\\H_n(t)&=n\sum_{i,j:t_{i,n}^{(1)} \wedge t_{j,n}^{(2)} \leq t} \big|\mathcal{I}_{i,n}^{(1)}\big|\big| \mathcal{I}_{j,n}^{(2)}\big| \mathds{1}_{\{\mathcal{I}_{i,n}^{(1)} \cap \mathcal{I}_{j,n}^{(2)}\neq \emptyset\}}+nO((|\pi_n|_t)^2)\overset{\mathbb{P}}{\longrightarrow}\big(\frac{2}{\lambda_1}+\frac{2}{\lambda_2}\big)t.
\end{aligned} 
\end{equation*}
\end{small}
Finally, we show that Condition \ref{cond_centr}(iii) is satisfied. Note first that the distributions of the sampling scheme $\pi_1$ and the rescaled $n \pi_n$ are identical. Therefore, the distributions of $n \eta_n^{(l)}(s)$ and $\eta_1^{(l)}(ns)$ are identical, and the distribution of the latter only depends on $s$ through the fact that the backward waiting time for the previous observation is bounded by $ns$. This effect becomes asymptotically irrevelant as $n$ grows, thus $n \eta_n^{(l)}(s)$ converges. Note also that the  $n \eta_n^{(l)}(s)$ are asymptotically independent because the Wiener process $W$ and the Poisson processes have independent increments and the $n \eta_n^{(l)}(S_p)$ overlap asymptotically with diminishing probability. Therefore the factorization of the expectations in \eqref{cond_centr_iii2} holds.

By symmetry we focus on $\eta^{(1)}(s)$ only which can be constructed from elementary distributions. Let $E_1^{(1)},E_2^{(1)} \sim Exp(\lambda_{1})$ and $E_1^{(2)},E_2^{(2)} \sim Exp(\lambda_{2})$ be independent. Then, after rescaling, the length of the interval around $s$ in the second series is asymptotically $E_1^{(2)} + E_2^{(2)}$, and conditionally the number of observations in the first series, which we will denote by $P$, is $Poisson (\lambda_1 (E_1^{(2)}+E_2^{(2)}))$. Then, if $(U_n)_{n \in \mathbb{N}}$ and $(N_n)_{n \in \mathbb{N}}$ are i.i.d.\ $\mathcal{U}[0,1]$ and $\mathcal{N}(0,1)$ random variables, respectively, it is easy to deduce that
\begin{align}\label{poiss1_representation}
\eta^{(1)}(s)\overset{\mathcal{L}}{=}\sum_{j=1}^{P+1} \left(R_{(j)}-R_{(j-1)} \right) \left(N_j \right)^2
\end{align}
holds, where we set 
\begin{align*}
&R_0=0 ,
\\&R_j=E_1^{(1)}+U_j(E_1^{(2)}+E_2^{(2)}), \quad j=1,\ldots,P,
\\&R_{P+1}=E_1^{(1)}+(E_1^{(2)}+E_2^{(2)})+E_2^{(1)}, 
\end{align*}
and let $R_{(j)}$ denote the $j$-th largest element from $R_0,\ldots,R_{P+1}$. 
\end{example}

\section{Testing for disjoint jumps} \label{sec:test}
\def\theequation{4.\arabic{equation}}
\setcounter{equation}{0}

We will introduce a test which makes use of a critical region of the form \eqref{critical_region}. In Section \ref{sec:clt} we have derived a central limit theorem for $\widetilde{\Phi}_{n,T}^{(d)}$. However, this result can not directly be applied for determining $\mathpzc{c}_n$ since the law of the limiting variable in Theorem \ref{clt} is itself random and not known to the statistician. Hence, in order to develop a statistical test we need to estimate the law of the limiting variable $\widetilde{\Psi}_T$. 

Estimating the continuous term $\widetilde{C}_T$ in $\widetilde{\Psi}_T$ boils down to estimating the continuous part of $X$. This can be done using truncated increments as in (4.5) of \cite{JacTod09}. With $\beta >0$ and $\varpi \in (0,1/2)$ we set
\begin{multline*}
A_{n,T}(\beta,\varpi)= n \sum_{i,j:t_{i,n}^{(1)} \wedge t_{j,n}^{(2)} \leq T} (\Delta_{i,n}^{(1)} X^{(1)})^2 (\Delta_{j,n}^{(2)} X^{(2)})^2 
\\ \times
\mathds{1}_{\{|\Delta_{i,n}^{(1)} X^{(1)} |\leq \beta |\mathcal{I}_{i,n}^{(1)}|^{\varpi} \wedge |\Delta_{j,n}^{(2)} X^{(2)} |\leq \beta |\mathcal{I}_{j,n}^{(2)}|^{\varpi}\}}
 \mathds{1}_{\{\mathcal{I}_{i,n}^{(1)} \cap \mathcal{I}_{j,n}^{(2)} \neq \emptyset\}}.
\end{multline*}
In order to estimate the law of $\widetilde{D}_T$ in $\widetilde{\Psi}_T$ we need to estimate the law of the $\eta^{(l)}(s)$ first, which is not known in practice unless one imposes knowledge on the nature of the sampling scheme. In principle, we would like to introduce a Monte Carlo approach and simulate the quantiles of $\widetilde{D}_T$, and a first approach obviously is to replace the increments of the Brownian motion in \eqref{def_eta_m} by appropriately scaled realizations of standard normal random variables. However, we have to scale by the lengths of the observation intervals which follow an unknown distribution. To circumvent this issue we use a bootstrap method and estimate the distribution of the observation intervals as well. The idea here is to estimate this distribution around time $s$ by using the observation interval which contains $s$ together with the $K_n$ previous and following intervals. In order for this procedure to work we will introduce a local homogeneity condition later, in the sense that $n\eta_n^{(l)}(s_1)$ and $n\eta_n^{(l)}(s_2)$ have similar distributions for small values of $|s_1-s_2|$ but become asymptotically independent otherwise. 

To formalize, let $(K_n)_n$ and $(M_n)_n$ denote deterministic sequences of integers which tend to infinity. For any $s \in (0,T)$ we define the random variables
\begin{align*}
\hat{\eta}_{n,m}^{(l)}(s)=n\sum_{i:\mathcal{I}_{i,n}^{(l)} \leq T} |\mathcal{I}_{i,n}^{(l)}|(U^{(l)}_{n,i,m})^2 
\mathds{1}_{\{\mathcal{I}_{i,n}^{(l)} \cap \mathcal{I}_{i_n^{(3-l)}(s)+V^{(3-l)}_{n,m}(s),n}^{(3-l)} \neq \emptyset\}}, ~~ l=1,2,
\end{align*}
the $U^{(l)}_{n,i,m}$ are $\mathcal{N}(0,1)$ random variables and the $V_{n,m}^{(l)}(s)$ are distributed according to
\begin{align*}
\widetilde{\mathbb{P}}(V^{(l)}_{n,m}(s)=k|\mathcal{S})= | \mathcal{I}_{i_n^{(l)}(s)+k,n}^{(l)}| \big(\sum_{j=-K_n}^{K_n} | \mathcal{I}_{i_n^{(l)}(s)+k,n}^{(l)}| \big)^{-1}, ~~ k\in\{-K_n,\ldots,K_n\}, 
\end{align*}
all $\mathcal S$-conditionally independent as $m = 1, \ldots, M_n$ varies. Both the $U^{(l)}_{n,i,m}$ and the $V_{n,m}^{(l)}(s)$ are defined on $(\widetilde{\Omega},\widetilde{\mathcal{A}},\widetilde{\mathbb{P}})$ as well. By construction, $\hat{\eta}_{n,m}^{(l)}(s)$ corresponds to a mixture of the $$n\eta_n^{(l)}\big(t_{i_n^{(3-l)}(s)+k,n}^{(3-l)}\big), \quad k\in\{-K_n,\ldots,K_n\},$$ where the increments of $W$ are replaced by the $U_{n,i,m}^{(l)}$ and the probability of choosing a specific $k$ is proportional to the length of $\mathcal{I}_{i_n^{(l)}(s)+k,n}^{(l)}$. This makes sense intuitively, as the probability of a jump of $X$ to fall into a specific interval is proportional to the length of the latter as well.

Consistent estimators for the jumps $\Delta X_s^{(l)}$ and the volatility $\big(\sigma_s^{(l)}\big)^2$, $l=1,2$, are given by
\begin{align*}
&\widehat{\Delta}_n X^{(l)}(s)=\Delta_{i_n^{(l)}(s),n}^{(l)} X^{(l)}  \mathds{1}_{\{|\Delta_{i_n^{(l)}(s),n}^{(l)}X^{(l)} |> \beta |\mathcal{I}_{i_n^{(l)}(s),n}^{(l)}|^{\varpi} \}} , \nonumber
\\&\big(\hat{\sigma}_n^{(l)}(s) \big)^2=\frac{1}{2b_n} 
\sum_{i:t_{i,n}^{(l)} \in [s-b_n,s+b_n]} \big(\Delta_{i,n}^{(l)} X^{(l)}\big)^2, \quad \text{if } \Delta X^{(l)}_s=0, 
\\ &\big(\tilde{\sigma}_n^{(l)}(s) \big)^2=\frac{1}{2b_n}
\sum_{i:t_{i,n}^{(l)} \in [s-b_n,s+b_n]} \big(\Delta_{i,n}^{(l)} X^{(l)}\big)^2 \mathds{1}_{\{|\Delta_{i,n}^{(l)} X^{(l)} |\leq \beta |\mathcal{I}_{i,n}^{(l)}|^{\varpi} \}}, \nonumber
\end{align*}
where $\beta>0$ and $\varpi \in (0,1/2)$, and $b_n$ is a sequence with $b_n \rightarrow 0$ and $|\pi_n|_T/b_n \overset{\mathbb{P}}{\longrightarrow}0$. In fact, we will use both estimators for $\sigma^{(l)}$ throughout the course of the paper, since we are only interested in estimating $\sigma^{(l)}$ when $X^{(3-l)}$ jumps, in which case $\Delta X^{(l)}$ vanishes under the null hypothesis.

Using these estimators we define 
\begin{multline*}
\widehat{D}_{T,n,m}=
\sum_{i,j:t_{i,n}^{(1)} \wedge t_{j,n}^{(2)} \leq T} \bigg(
\big(\Delta_{i,n}^{(1)} X^{(1)} \big)^2 \mathds{1}_{\{|\Delta_{i,n}^{(1)} X^{(1)} |> \beta |\mathcal{I}_i^n|^{\varpi} \}}\big(\hat{\sigma}_n^{(2)}(t_{i,n}^{(1)}) \big)^2\hat{\eta}_{n,m}^{(2)}(t_{i,n}^{(1)})
\\+\big(\Delta_{j,n}^{(2)} X^{(2)} \big)^2 \mathds{1}_{\{|\Delta_{j,n}^{(2)} X^{(2)} |> \beta |\mathcal{J}_j^n|^{\varpi} \}}\big(\hat{\sigma}_n^{(1)}(t_{j,n}^{(2)}) \big)^2 \hat{\eta}_{n,m}^{(1)}(t_{j,n}^{(2)}) \bigg),
\end{multline*}
and for $\alpha \in [0,1]$ we set
\begin{align*}
\widehat{Q}_{n,T}(\alpha)=\widehat{Q}_{\alpha}\big(\big\{\widehat{D}_{T,n,m} \big|m=1,\ldots,M_n \big\}\big)
\end{align*}
where $\widehat{Q}_{\alpha}(B)$ denotes the $\lfloor \alpha N \rfloor$-th largest element of a set $B$ with $N \in \mathbb{N}$ elements. $\widetilde{D}_{T,n,m}$ and $\widetilde{Q}_{n,T}(\alpha)$ are defined analogously by replacing $\hat{\sigma}_n^{(l)}$ with $\tilde{\sigma}_n^{(l)}$. We will see that these expressions consistently estimate the $\mathcal{X}$-conditional $\alpha$ quantile of $\widetilde{D}_T$.

The following condition summarizes all additional assumptions we need in order to obtain an asymptotic test. It ensures in particular that the empirical common distribution of the $\hat{\eta}^{(l)}_{n,m}(s_j)$ converges to the common distribution of the $\eta^{(l)}(s_j)$ which is essential for the bootstrap method to work. 

\begin{condition}\label{condition_test}
The process $X$ and the sequence of observation schemes $(\pi_n)_n$ satisfy Condition \ref{cond_centr}, and $(b_n)_n$ fulfills $|\pi_n|_T/b_n \overset{\mathbb{P}}{\longrightarrow}0$. Also, $(K_n)_n$ and $(M_n)_n$ are sequences of integers converging to infinity, and $|\pi_n|_T K_n \overset{\mathbb{P}}{\longrightarrow}0$. Additionally,
\begin{align} \label{vert_kon1}
\widetilde{\mathbb{P}}\big(\big| \widetilde{\mathbb{P}}\big(\hat{\eta}^{(l_j)}_{n,1}(s_j)\leq x_j,~j=1,\ldots,J \big| \mathcal S \big)
-
  \widetilde{\mathbb{P}}\big({\eta}^{(l_j)}(s_j)\leq x_j,~j=1,\ldots,J\big) \big|> \varepsilon \big) \to 0
\end{align}
%\begin{multline}\label{vert_kon1}
%\widetilde{\mathbb{P}}\big(\big|\frac{1}{M_n}  \sum_{m=1}^{M_n} \big(\mathds{1}_{\{\hat{\eta}^{(l_j)}_{n,m}(s_j)\leq x_j,~j=1,\ldots,J\}} 
%- \widetilde{\mathbb{P}}(\eta^{(l_j)}(s_j) \leq x_j,~j=1,\ldots,J)\big) \big|> \varepsilon \big) \rightarrow 0
%\end{multline}
as $n \rightarrow \infty$, for all $\varepsilon > 0$, $J \in \mathbb{N}$, $x=(x_1, \ldots, x_J) \in \R^J$, $l_j\in\{1,2\}$ and $s_j \in (0,T)$, $j=1,\ldots,J$, with $s_i \neq s_j$ for $i \neq j$.
\end{condition}
Let either $Q_{n,T}(1-\alpha)=\widehat{Q}_{n,T}(1-\alpha)$ or $Q_{n,T}(1-\alpha)=\widetilde{Q}_{n,T}(1-\alpha)$.

\begin{theorem}\label{test_theo}
If Condition \ref{condition_test} is satisfied, the test defined in \eqref{critical_region} with
\begin{align*}
\mathpzc{c}_n= \frac{A_{n,T}(\beta, \varpi)+Q_{n,T}(1-\alpha)}{n\sqrt{V^{(1)}(g,\pi_n)_TV^{(2)}(g,\pi_n)_T}}, \quad \alpha \in [0,1],
\end{align*}
has asymptotic level $\alpha$ in the sense that we have
\begin{align}\label{test_theo_level}
\widetilde{\mathbb{P}}\big(\widetilde{\Phi}_{n,T}^{(d)} > \mathpzc{c}_n \big| F^{(d)} \big) \rightarrow \alpha 
\end{align}
for all $F^{(d)} \subset \Omega_T^{(d)}$ with $\mathbb{P}(F^{(d)})>0$. Because of
\begin{align}\label{test_theo_power}
\widetilde{\mathbb{P}}\big(\widetilde{\Phi}_{n,T}^{(d)} > \mathpzc{c}_n \big| F^{(j)} \big) \rightarrow 1 
\end{align}
for all $F^{(j)} \subset \Omega_T^{(j)}$ with $\mathbb{P}( F^{(j)})>0$ it is consistent as well.
\end{theorem}

\begin{example}
If the sampling scheme is deterministic, then \eqref{vert_kon1} holds in all situations where a minimal local regularity is assumed. This is in particular the case for the setting of synchronous equidistant observation times as in Example \ref{ext_equid} where our estimator $\widetilde{Q}_{n,T}(\alpha)$ equals the estimator $Z_n^{(d)}(\alpha)$ defined in (5.10) of \cite{JacTod09} for $N_n=M_n$ and any choice of $K_n$ (not necessarily converging to infinity).
\end{example}

\begin{example}\label{poiss2} 
Regarding the Poisson setting from Example \ref{poiss1}, $|\pi_n|_T/b_n \overset{\mathbb{P}}{\longrightarrow}0$ follows from \eqref{Poisson_mesh} for every $b_n=O(n^{-\alpha})$ with $\alpha \in (0,1)$. Showing that \eqref{vert_kon1} holds, however, is rather tedious and postponed to Section \ref{sec:proof}. 
\end{example}

\section{Simulation results} \label{sec:simul}
\def\theequation{5.\arabic{equation}}
\setcounter{equation}{0}

We conduct a simulation study to verify the finite sample properties of the introduced methods. Our benchmark model is the one from Section 6 of \cite{JacTod09}, as we use the same configuration as in their paper to compare our approach to the case of equidistant and synchronous observations. The model for $X$ is given by
\begin{equation*}
\begin{aligned}
&dX^{(1)}_t=X^{(1)}_t \sigma_1 d W^{(1)}_t+\alpha_1 \int_{\R} X^{(1)}_{t-}x_1 \mu_1(dt,dx_1)+\alpha_3 \int_{\R} X^{(1)}_{t-}x_3 \mu_3(dt,dx_3),
\\&dX^{(2)}_t=X^{(2)}_t \sigma_2 d W^{(2)}_t+\alpha_2 \int_{\R} X^{(2)}_{t-}x_2 \mu_2(dt,dx_2)+\alpha_3 \int_{\R} X^{(2)}_{t-}x_3 \mu_3(dt,dx_3),
\end{aligned}
\end{equation*}
where $[W^{(1)},W^{(2)}]_t=\rho t$ and the Poisson measures $\mu_i$ are independent of each other and have predictable compensators $\nu_i$ of the form
\[
\nu_i(dt,dx_i)=\kappa_i\frac{\mathds{1}_{[-h_i,-l_i]\cup [l_i,h_i]}(x_i)}{2(h_i-l_i)}dtdx_i
\]
where $0<l_i<h_i$ for $i=1,2,3$, and the initial values are $X_0=(1,1)^T$. We consider the same twelve parameter settings which were discussed in \cite{JacTod09} of which six allow for common jumps and six do not. In the case where common jumps are possible, we only use the simulated paths which contain common jumps. For the parameters we set $\sigma_1^2=\sigma_2^2=8 \times 10^{-5}$ in all scenarios and choose the parameters for the Poisson measures such that the contribution of the jumps to the total variation remains approximately constant and matches estimations from real financial data (see \cite{HuaTau06}). The parameter settings are summarized in Table \ref{par_settings}.

\begin{table}[b]
\centering
\resizebox{13.3cm}{!} {
\begin{tabular}{lccccccccccccc}
\hline 
 & \multicolumn{13}{c}{Parameters} \\
\cline{2-14}
Case & $\rho$ & $\alpha_1$ & $\kappa_1$ & $l_1$ & $h_1$ & $\alpha_2$ & $\kappa_2$ & $l_1$ & $h_1$ & $\alpha_3$ & $\kappa_3$ & $l_3$ & $h_3$ \\ 
\hline 
I-j & $0.0$ & $0.00$ & • & • & • & $0.00$ & • & • & • & $0.01$ & $1$ & $0.05$ & $0.7484$ \\ 

II-j & $0.0$ & $0.00$ & • & • & • & $0.00$ & • & • & • & $0.01$ & $5$ & $0.05$ & $0.3187$ \\ 

III-j & $0.0$ & $0.00$ & • & • & • & $0.00$ & • & • & • & $0.01$ & $25$ & $0.05$ & $0.1238$ \\ 

I-m & $0.5$ & $0.01$ & $1$ & $0.05$ & $0.7484$ & $0.01$ & $1$ & $0.05$ & $0.7484$ & $0.01$ & $1$ & $0.05$ & $0.7484$ \\ 

II-m & $0.5$ & $0.01$ & $5$ & $0.05$ & $0.3187$ & $0.01$ & $5$ & $0.05$ & $0.3187$ & $0.01$ & $5$ & $0.05$ & $0.3187$ \\ 

III-m & $0.5$ & $0.01$ & $25$ & $0.05$ & $0.1238$ & $0.01$ & $25$ & $0.05$ & $0.1238$ & $0.01$ & $25$ & $0.05$ & $0.1238$ \\ 

I-d0 & $0.0$ & $0.01$ & $1$ & $0.05$ & $0.7484$ & $0.01$ & $1$ & $0.05$ & $0.7484$ & • & • & • & • \\ 

II-d0 & $0.0$ & $0.01$ & $5$ & $0.05$ & $0.3187$ & $0.01$ & $5$ & $0.05$ & $0.3187$ & • & • & • & • \\ 

III-d0 & $0.0$ & $0.01$ & $25$ & $0.05$ & $0.1238$ & $0.01$ & $25$ & $0.05$ & $0.1238$ & • & • & • & • \\ 

I-d1 & $1.0$ & $0.01$ & $1$ & $0.05$ & $0.7484$ & $0.01$ & $1$ & $0.05$ & $0.7484$ & • & • & • & • \\ 

II-d1 & $1.0$ & $0.01$ & $5$ & $0.05$ & $0.3187$ & $0.01$ & $5$ & $0.05$ & $0.3187$ & • & • & • & • \\ 

III-d1 & $1.0$ & $0.01$ & $25$ & $0.05$ & $0.1238$ & $0.01$ & $25$ & $0.05$ & $0.1238$ & • & • & • & • \\ 
\hline 
\end{tabular} }
\caption[]{Parameter settings for the simulation.}
\label{par_settings}
\end{table}

\begin{figure}[t]
\centering
\includegraphics[width=\textwidth]{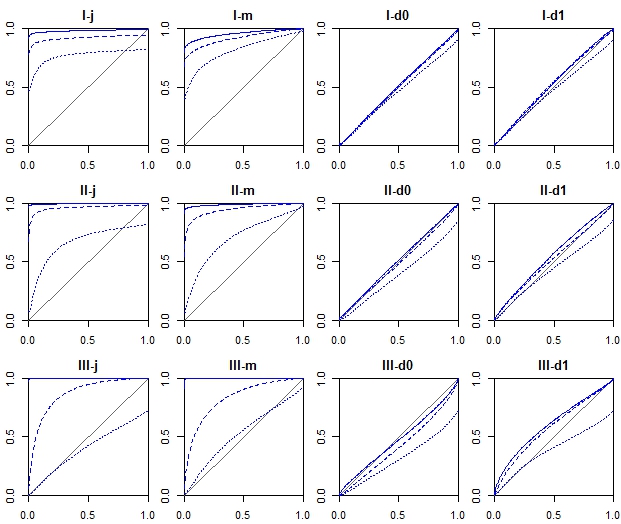}
\caption{Empirical rejection curves from the Monte Carlo simulation for the test derived from Theorem \ref{test_theo}. The dotted line represents the results for $n=100$, the dashed line for $n=400$ and the solid line for $n=1600$. In each case $N=10.000$ paths were simulated.}
\label{rejection_curves}
\end{figure}

To model the observation times we use the Poisson setting discussed in Example \ref{poiss1} and \ref{poiss2} for $\lambda_1=1$ and $\lambda_2=2$, and set $T=1$ which amounts to approximately $n$ observations of $X^{(1)}$ and $2n$ observations of $X^{(2)}$. We choose $n=100,n=400$ and $n=1600$ for the simulation. In a trading day of $6.5$ hours this corresponds to observing $X^{(1)}$ on average every $4$ minutes, every $1$ minute and every $15$ seconds.
 We set $\beta=0.03$ and $\varpi=0.49$ for all occuring truncations. We use $\hat{\sigma}^{(l)}(s)$ as an estimator for $\sigma^{(l)}_s$, $b_n= 1/\sqrt{n}$ for the local interval in the estimation of $\sigma_s^{(l)}$ and $K_n=\lfloor \ln (n)\rfloor$, $M_n=n$ in the simulation of the $\hat{\eta}^{(l)}_{n,m}(s)$.

In Figure \ref{rejection_curves} we display the results from the simulation. The plots are constructed as follows: First for different values of $\alpha$ the critical values are simulated according to Theorem \ref{test_theo}. Then we plot the observed rejection frequencies over $\alpha$. 

The six plots on the left show the results for the cases where the alternative of common jumps is true. In the cases I-j, II-j and III-j there exist only joint jumps and the Brownian motions $W^{(1)}$ and $W^{(2)}$ are uncorrelated. In the cases I-m, II-m and III-m we have a mixed model which allows for disjoint and joint jumps and also the Brownian motions are positively correlated. The prefixes I, II and III indicate an increasing number of jumps present in the observed paths. Since our choice of parameters is such that the overall contribution of the jumps to the quadratic variation is roughly the same in all parameter settings, this corresponds to a decreasing size of the jumps. Hence in the cases I-* we have few big jumps while in the cases III-* we have many small jumps.

We see that the test has very good power against the alternative of common jumps. The power is greater for small $n$ if there are less and bigger jumps as can be seen from the dotted lines for the cases I-j and I-m, because the bigger jumps are detected more easily. On the other hand the power is greater for large $n$ if there are more and smaller jumps which can be seen from the solid lines for III-j and III-m, because then it is more probable that at least one of the common jumps is detected and one small detected common jump is sufficient for rejecting the null. 

The six plots on the right in Figure \ref{rejection_curves} show the results for the cases where the null hypothesis is true. While in the cases *-d0 the Brownian motions $W^{(1)}$ and $W^{(2)}$ are uncorrelated, the Brownian motions are perfectly correlated in the cases *-d1. The prefixes I, II and III stand for an increasing number and a decreasing size of the jumps as in the first six cases. 

Under the null of disjoint jumps we see that the observed rejection frequencies match the predicted asymptotic rejection probabilities from Theorem \ref{test_theo} very well in all six cases. There are slight deviations for a higher number of jumps. This is due to the fact that disjoint jumps whenever they lie close together, sometimes cannot be distinguished based on the observations which leads to over-rejection under the null hypothesis. In the cases *-d1 where the Brownian motions are perfectly correlated the rejection frequencies are systematically too high for large $n$. The results are worse than in the cases *-d0. 

In general, the results from the Monte Carlo match the results from \cite{JacTod09} very closely and we even receive slightly better results in the cases I-d1, II-d1 and especially in case III-d1. This is of great importance, as these results demonstrate that it is possible to construct a test for disjoint jumps which works efficiently in the case of asynchronous and random observations without having to synchronize data first. Such procedures are well-known in the literature, but lead inevitably to a loss of data and, thus, power. Also, our methods are applicable in a quite universal setting without additional knowledge on the underlying observation scheme. 

\section{Proofs} \label{sec:proof}
\def\theequation{6.\arabic{equation}}
\setcounter{equation}{0}

\subsection{Preliminaries}

Throughout the proofs we will assume that the processes $b_s,\sigma_s^{(1)},\sigma_s^{(2)},\rho_s$ and $s \mapsto \delta(s,z)$ are bounded on $[0,T]$. They are continuous by Condition \ref{cond_consistency} and therefore locally bounded. A localization procedure then shows that the results for bounded processes can be carried over to the case of locally bounded processes (see e.g. Section 4.4.1 in \cite{JacPro12}).

We introduce the decomposition $X_t=X_0+B(q)_t+C_t+M(q)_t+N(q)_t$ of the It\^o semimartingale \eqref{ItoSemimart} with
\begin{align*}
B(q)_t&=\int_0^t \big( b_s -\int(\delta(s,z)\mathds{1}_{\{\|\delta\|\leq 1\}}-\delta(s,z)\mathds{1}_{\{\gamma(z)\leq1/q\}})\lambda(dz)\big)ds,
\\C_t &= \int_0^t \sigma_s dW_s,
\\M(q)_t&=\int_0^t \int \delta(s,z) \mathds{1}_{\{\gamma(z)\leq 1/q\}}(\mu-\nu)(ds,dz),
\\N(q)_t&=\int_0^t \int \delta(s,z) \mathds{1}_{\{\gamma(z)>1/q\}}\mu(ds,dz).
\end{align*}
Here $q$ is a parameter which controls whether jumps are classified as small jumps or big jumps. We will make repeatedly use of the following estimates (compare Section 2.1.5 in \cite{JacPro12}).
\begin{lemma}\label{elem_ineq}
There exist constants $K,K_p, K_q,e_q \geq 0$ such that
\begin{align}
&\|B(q)_{s+t}-B(q)_s\|^2 \leq K_q t^2, \nonumber
\\ &\mathbb{E} \big[ \|C_{s+t}-C_{s}\|^p|\mathcal{F}_s\big] \leq K_p t^{p/2}, \label{elem_ineq_C}
\\ &\mathbb{E} \big[ \|M(q)_{s+t}-M(q)_{s}\|^2|\mathcal{F}_s\big] \leq K t e_q, \nonumber
\\ &\mathbb{E} \big[ \|N(q)_{s+t}-N(q)_{s}\|^2|\mathcal{F}_s\big] \leq K_q t, \nonumber
\end{align}
for all $s , t \geq 0$, $q >0$, $p \geq 1$. Here, $e_q$ can be chosen such that $e_q \rightarrow 0$ for $q \rightarrow \infty$.
\end{lemma}
Throughout the proofs $K$ and $K_q$ will denote generic constants, the latter dependent on $q$, to simplify notation. 

\subsection{Proof of the consistency result}

\begin{proof}[Proof of \eqref{cons_Vf}]
We will show
\begin{gather}\label{proof_cons1}
\lim_{q \rightarrow \infty} \limsup_{n \rightarrow \infty} \mathbb{P} \big(\big| 
\sum_{i,j:t_{i,n}^{(1)} \wedge t_{j,n}^{(2)}\leq T} \big(\Delta_{i,n}^{(1)} N^{(1)}(q)\Delta_{j,n}^{(2)} N^{(2)}(q) \big)^2\mathds{1}_{\{\mathcal{I}_{i,n}^{(1)} \cap \mathcal{I}_{j,n}^{(2)}\neq \emptyset \}}
-B_T\big|>\delta \big)\rightarrow 0 
\end{gather}
and
\begin{gather}\label{proof_cons2}
\lim_{q \rightarrow \infty} \limsup_{n \rightarrow \infty} \mathbb{P} \big(\big| V(f,\pi_n)_T-
\sum_{i,j:t_{i,n}^{(1)} \wedge t_{j,n}^{(2)}\leq T} \big(\Delta_{i,n}^{(1)} N^{(1)}(q)\Delta_{j,n}^{(2)} N^{(2)}(q) \big)^2\mathds{1}_{\{\mathcal{I}_{i,n}^{(1)} \cap \mathcal{I}_{j,n}^{(2)}\neq \emptyset \}}
\big|>\delta \big)\rightarrow 0
\end{gather}
for all $\delta>0$ from which \eqref{cons_Vf} follows. 

For proving \eqref{proof_cons1} we denote by $\Omega(n,q)$ the set on which two different jumps of $N(q)$ are further apart than $2|\pi_n|_T$. On $\Omega(n,q)$ we have
\begin{gather}\label{proof_cons1.1}
\sum_{i,j:t_{i,n}^{(1)} \wedge t_{j,n}^{(2)}\leq T} \big(\Delta_{i,n}^{(1)} N^{(1)}(q)\Delta_{j,n}^{(2)} N^{(2)}(q) \big)^2\mathds{1}_{\{\mathcal{I}_{i,n}^{(1)} \cap \mathcal{I}_{j,n}^{(2)}\neq \emptyset \}}
=\sum_{s \leq T} \big(\Delta N^{(1)}(q)_s \big)^2 \big(\Delta N^{(2)}(q)_s \big)^2.
\end{gather}
Note that the right hand side of \eqref{proof_cons1.1} converges to $B_T$ for $q\rightarrow \infty$. Thus, \eqref{proof_cons1} follows since $\mathbb{P}(\Omega(n,q))\rightarrow 1$ for $n \rightarrow \infty$. 

For proving \eqref{proof_cons2} we introduce the elementary inequality 
\begin{multline}\label{algebr_ineq}
 \big| \big( a_1+b_1+c_1+d_1 \big)^2 (a_2+b_2+c_2+d_2)^2-{d_1}^2 {d_2}^2 \big|
\\ \leq c_{\rho} \sum_{l=1,2} \big( {a_{3-l}}^2+{b_{3-l}}^2+{c_{3-l}}^2 \big)\big( {a_l}^2+{b_l}^2+{c_l}^2+{d_l}^2 \big)+ 3 \rho {d_1}^2 {d_2}^2
\end{multline}
which can be proven using Cauchy-Schwarz inequality after introducing appropriate weights and holds for real numbers $a_l,b_l,c_l,d_l\in \mathbb{R}$, $l=1,2$, and $\rho \in (0,1)$ by setting $c_{\rho}=9 \big(1+\rho \big)^2/\rho^2$. As we are interested in the sum of the product of the squared increments of $X^{(1)}$ and $X^{(2)}$, we can simplify each summand by applying (\ref{algebr_ineq}), i.e.\ we set   $a_l=\Delta_{i,n}^{(l)}B^{(l)}(q),b_l=\Delta_{i,n}^{(l)}C^{(l)},c_l=\Delta_{i,n}^{(l)}M^{(l)}(q),d_l=\Delta_{i,n}^{(l)}N^{(l)}(q)$.

Note that
\[
3 \rho \sum_{i,j:t_{i,n}^{(1)} \wedge t_{j,n}^{(2)}\leq T} \big(\Delta_{i,n}^{(1)} N^{(1)}(q)\Delta_{j,n}^{(2)} N^{(2)}(q) \big)^2\mathds{1}_{\{\mathcal{I}_{i,n}^{(1)} \cap \mathcal{I}_{j,n}^{(2)}\neq \emptyset \}}
 \rightarrow 3 \rho \big[N^{(1)}(q),N^{(2)}(q)  \big]_T
\]
which tends to zero for $\rho \rightarrow 0$. Furthermore, for any $l=1,2$,
\begin{small}
\begin{align*}
&c_\rho \sum_{i,j:t_{i,n}^{(3-l)} \wedge t_{j,n}^{(l)}\leq T} \big(\big(\Delta_{i,n}^{(3-l)} B^{(3-l)}(q)\big)^2 +\big(\Delta_{i,n}^{(3-l)} M^{(3-l)}(q) \big)^2\big) \times \mathds{1}_{\{\mathcal{I}_{i,n}^{(3-l)} \cap \mathcal{I}_{j,n}^{(l)}\neq \emptyset \}}
\\ &~~~~\times
\big(\big(\Delta_{j,n}^{(l)} B^{(l)}(q)\big)^2+
\big(\Delta_{j,n}^{(l)} C^{(l)}\big)^2+ \big(\Delta_{j,n}^{(l)} M^{(l)}(q) \big)^2+\big(\Delta_{j,n}^{(l)} N^{(l)}(q) \big)^2\big)
\\ &~~\leq c_\rho \big( \sum_{i:t_{i,n}^{(3-l)}\leq T} \big(\big(\Delta_{i,n}^{(3-l)} B^{(3-l)}(q)\big)^2 +\big(\Delta_{i,n}^{(3-l)} M^{(3-l)}(q) \big)^2\big) \big)
\\&~~~~ \times \big(\sum_{j:t_{i,n}^{(l)}\leq T}\big( \big(\Delta_{j,n}^{(l)} B^{(l)}(q)\big)^2+
\big(\Delta_{j,n}^{(l)} C^{(l)}\big)^2+ \big(\Delta_{j,n}^{(l)} M^{(l)}(q) \big)^2+\big(\Delta_{j,n}^{(l)} N^{(l)}(q) \big)^2   \big)
\\&~~ \overset{\mathbb{P}}{\longrightarrow} c_\rho\big( [B^{(3-l)}(q),B^{(3-l)}(q)]_T+ [M^{(3-l)}(q),M^{(3-l)}(q)]_T
\big)\big[X^{(l)},X^{(l)}\big]_T
\end{align*}
\end{small}which tends to zero for $q \rightarrow \infty$. For the remaining terms we set
\begin{align*}
K(l,\varepsilon)= \sup_{0\leq t_0 \leq t_1 \leq \ldots \leq t_m \leq T, |t_m-t_0| \leq \varepsilon} \sum_{k=1}^m \big(C^{(l)}_{t_{k}}-C^{(l)}_{t_{k-1}} \big)^2, \quad l=1,2.
\end{align*}
We have $K(l,\varepsilon)\overset{\mathbb{P}}{\longrightarrow} 0$ for $\varepsilon \rightarrow 0$ due to the ucp convergence of realized volatility to the quadratic variation. 
Using the fact that the total length of the observation intervals of one process which overlap with a specific observation interval of the other process is at most $3|\pi_n|_T$, we get on the set $\{|\pi_n|_T\leq \varepsilon\} $
\begin{small}
\begin{align*}
c_\rho& \sum_{i,j:t_{i,n}^{(3-l)} \wedge t_{j,n}^{(l)}\leq T} \big(\Delta_{i,n}^{(3-l)} C^{(3-l)}\big)^2 \big(\big(\Delta_{j,n}^{(l)} C^{(l)} \big)^2 +\big(\Delta_{j,n}^{(l)} N^{(l)}(q) \big)^2  \big) \mathds{1}_{\{\mathcal{I}_{i,n}^{(3-l)} \cap \mathcal{I}_{j,n}^{(l)}\neq \emptyset \}} 
\\&\leq c_\rho K(3-l,3\varepsilon) \sum_{j: t_{j,n}^{(l)}\leq T} \big(\big(\Delta_{j,n}^{(l)} C^{(l)} \big)^2 +\big(\Delta_{j,n}^{(l)} N^{(l)}(q) \big)^2  \big). 
\end{align*}
\end{small}As the latter sum converges to the quadratic variation of $C^{(l)}+N^{(l)}$, we obtain that these terms vanish as well since $K(3-l,3\varepsilon)\overset{\mathbb{P}}{\longrightarrow} 0$ for $\varepsilon \rightarrow 0$ and $\mathbb{P}(|\pi_n|_T\leq \varepsilon)\rightarrow 1$ as $n \rightarrow \infty$ for any fixed $\varepsilon>0$.
\end{proof}

\begin{proof}[Proof of Theorem \ref{cons_theo}] This is a direct consequence of \eqref{cons_Vf} and the continuous mapping theorem for convergence in probability, as \eqref{cons_Vf} implies \eqref{cons_Vg}.
\end{proof}

\subsection{Proof of the central limit theorem}
We will prove the central limit theorem in three parts: We will begin with the convergence of the mixed Brownian increments to the continuous term in the limit (Proposition \ref{clt_prop1}), followed by the convergence of the mixed term of large jumps and Brownian increments to the mixed term in the limit (Proposition \ref{clt_prop2}), and we end with the convergence of the remaining terms to zero (Proposition \ref{clt_prop3}).

\begin{proposition}\label{clt_prop1}
If Condition \ref{cond_centr}(i)-(ii) is fulfilled, we have
\begin{multline*}
n\sum_{i,j:t_{i,n}^{(1)} \wedge t_{j,n}^{(2)} \leq T} \big((\Delta_{i,n}^{(1)} C^{(1)} )^2(\Delta_{j,n}^{(2)} C^{(2)} )^2\big)
\mathds{1}_{\{\mathcal{I}_{i,n}^{(1)} \cap\mathcal{I}_{j,n}^{(2)} \neq \emptyset\}} 
\\ \overset{\mathbb{P}}{\longrightarrow}
\int_0^T \big(2(\rho_s\sigma_s^{(1)}\sigma_s^{(2)})^2G'(s) +(\sigma_s^{(1)}\sigma_s^{(2)})^2 H'(s) \big)ds.
\end{multline*}
\end{proposition}

\begin{proof}[Proof of Proposition \ref{clt_prop1}]
We use a discretization of $\sigma$ given via $\sigma(r)_s=\sigma_{(k-1)T/2^r}$ for $s \in [(k-1)T/2^r,kT/2^r)$, and we we denote the integral of $\sigma(r)$ with respect to the Brownian motion $W$ from \eqref{ItoSemimart} by $C(r)$. Setting
\begin{align*}
&R_n=n\sum_{i,j:t_{i,n}^{(1)} \wedge t_{j,n}^{(2)} \leq T} \big((\Delta_{i,n}^{(1)} C^{(1)} )^2(\Delta_{j,n}^{(2)} C^{(2)} )^2\big)
\mathds{1}_{\{\mathcal{I}_{i,n}^{(1)} \cap\mathcal{I}_{i,n}^{(2)} \neq \emptyset\}}, 
\\&R=\int_0^T \big(2(\rho_s\sigma_s^{(1)}\sigma_s^{(2)})^2G'(s) +(\sigma_s^{(1)}\sigma_s^{(2)})^2 H'(s) \big)ds,
\\&R_n(r)=n\sum_{i,j:t_{i,n}^{(1)} \wedge t_{j,n}^{(2)} \leq T} \big((\Delta_{i,n}^{(1)} C^{(1)}(r) )^2(\Delta_{j,n}^{(2)} C^{(2)}(r) )^2\big)
\mathds{1}_{\{\mathcal{I}_{i,n}^{(1)} \cap\mathcal{I}_{i,n}^{(2)} \neq \emptyset\}}, 
\\&R(r)=\int_0^T \big(2(\rho(r)_s\sigma^{(1)}(r)_s\sigma^{(2)}(r)_s)^2G'(s) +(\sigma^{(1)}(r)_s\sigma^{(2)}(r)_s)^2 H'(s) \big)ds,
\end{align*}
we will prove
\begin{align*}
\lim_{r \rightarrow \infty} \limsup_{n \rightarrow \infty} \mathbb{P} \big( \big|R-R(r) \big|+\big| R(r)-R_n(r)\big|+\big|R_n(r)-R_n \big|>\varepsilon  \big) = 0~~ \forall \varepsilon>0.
\end{align*}
By Condition \ref{cond_consistency}, $\sigma$ is uniformly continuous on $[0,T]$. Thus, $\sigma(r)$ converges uniformly to $\sigma$ for $r \rightarrow \infty$ on $[0,T]$, and we have $\big|R-R(r)\big|\rightarrow 0$ almost surely.

In order to prove $|R(r)-R_n(r)| \overset{\mathbb{P}}{\longrightarrow} 0$ as $n \rightarrow \infty$ we apply Lemma 2.2.12 from \cite{JacPro12} with 
\begin{align*}
\xi_k^n=n\sum_{(i,j) \in L(n,k,T) } \big((\Delta_{i,n}^{(1)} C^{(1)}(r) )^2(\Delta_{j,n}^{(2)} C^{(2)}(r) )^2\big)
\mathds{1}_{\{\mathcal{I}_{i,n}^{(1)} \cap\mathcal{I}_{i,n}^{(2)} \neq \emptyset\}}, 
\end{align*}
$L(n,k,T)=\{(i,j):t_{i-1,n}^{(1)} \vee t_{j-1,n}^{(2)}  \in  [(k-1)T/2^{r_n},kT/2^{r_n})\}$, $k=1,2,\ldots,2^{r_n}$, and $\mathcal{G}_k^n=\sigma\big(\mathcal{F}_{(k-1)T/2^{r_n}} \cup \mathcal{S}\big)$. Here, $r_n$ is a sequence of real numbers with $r_n \geq r$, $r_n \rightarrow \infty$ and
\begin{align}\label{r_n}
&2^{r_n} \sup_{s \in [0,T]} |G(s)-G_n(s)|=o_{\mathbb{P}}(1), \nonumber
\\& 2^{r_n} \sup_{s \in [0,T]} |H(s)-H_n(s)|=o_{\mathbb{P}}(1),
\\&2^{r_n}n (|\pi_n|_T)^2=o_{\mathbb{P}}(1). \nonumber
\end{align}
Such a sequence exists, because $G_n,H_n$ and hence $G,H$ are nondecreasing functions such that the pointwise convergence from Condition \ref{cond_centr}(ii) implies uniform convergence on $[0,T]$ and because of $n (|\pi_n|_T)^2=o_{\mathbb{P}}(1)$ by Condition \ref{cond_centr}(i). Elementary computations then reveal
\begin{align*}
&\mathbb{E}\big[\xi_k^n \big| \mathcal{G}_{k-1}^n  \big]=2(\rho(r)_{(k-1)T/2^{r_n}}\sigma^{(1)}(r)_{(k-1)T/2^{r_n}}\sigma^{(2)}(r)_{(k-1)T/2^{r_n}})^2  \\&~~~~~~~~~~~~~~~~~~~~~~~~~~~~~~~~~~~~~~~~~~~~\times \big( G_n(kT/2^{r_n})-G_n((k-1)T/2^{r_n}) \big) 
\\&~~~~+(\sigma^{(1)}(r)_{(k-1)T/2^{r_n}}\sigma^{(2)}(r)_{(k-1)T/2^{r_n}})^2 \big( H_n(kT/2^{r_n})-H_n((k-1)T/2^{r_n}) \big)
\\&~~~~+O_{\mathbb{P}}(n (|\pi_n|_T)^2).
\end{align*}
In combination with the boundedness of $\sigma$ the previous display implies
\begin{multline*}
\big|R(r)-\sum_{k=1}^{2^{r_n}} \mathbb{E}\big[\xi_k^n \big| \mathcal{G}_{k-1}^n  \big]\big|
\\\leq K2^{r_n}\big(\sup_{s \in [0,T]} |G(s)-G_n(s)|+\sup_{s \in [0,T]} |H(s)-H_n(s)|\big)+O_{\mathbb{P}}\big(  2^{r_n}n(|\pi_n|_T)^2\big)
\end{multline*}
where the right hand side is $o_{\mathbb{P}}(1)$ by \eqref{r_n}. Hence the sum over the $\mathbb{E}\big[\xi_k^n \big| \mathcal{G}_{k-1}^n  \big]$ converges to $R(r)$.

Using the Cauchy-Schwarz inequality, the definition of $H_n$ and telescoping sums we also get
\begin{align*}
\sum_{k=1}^{2^{r_n}}\mathbb{E}\big[\big|\xi_k^n\big|^2 \big| \mathcal{G}_{k-1}^n  \big] &\leq K \sum_{k=1}^{2^{r_n}} \Big( n\sum_{(i,j) \in L(n,k,T) }
\big| \mathcal{I}_{i,n}^{(1)} \big | \big| \mathcal{I}_{j,n}^{(2)} \big| 
\mathds{1}_{\{\mathcal{I}_{i,n}^{(1)} \cap\mathcal{I}_{i,n}^{(2)} \neq \emptyset\}} \Big)^2
\\ &\leq K H_n(T) \sup_{u,s \in [0,T], |u-s|\leq T 2^{-r_n}}\big|H_n(u)-H_n(s)|
\end{align*}
where the right hand side converges to zero in probability, since $H_n$ converges uniformly to a continuously differentiable function $H$. Together with 
$$
\sum_{k=1}^{2^{r_n}} \mathbb{E}\big[\xi_k^n \big| \mathcal{G}_{k-1}^n  \big] \overset{\mathbb{P}}{\longrightarrow} R(r)
$$
we obtain
$$
R_n(r)=\sum_{k=1}^{2^{r_n}} \xi_k^n  \overset{\mathbb{P}}{\longrightarrow} R(r)
$$
by Lemma 2.2.12 from \cite{JacPro12}.

Finally, we have
\begin{multline*}
\big|R_n(r)-R_n\big| \leq n 
\sum_{i,j:t_{i,n}^{(1)} \wedge t_{j,n}^{(2)} \leq T} \bigg(
\big|\Delta_{i,n}^{(1)}\big(C^{(1)}-C^{(1)}(r)\big)\big|^2\big|\Delta_{j,n}^{(2)}C^{(2)}  \big|^2 \\
+\big|\Delta_{j,n}^{(2)}\big(C^{(2)}-C^{(2)}(r)\big)\big|^2\big|\Delta_{i,n}^{(1)}C^{(1)}  \big|^2
\bigg)
\mathds{1}_{\{\mathcal{I}_{i,n}^{(1)} \cap\mathcal{I}_{i,n}^{(2)} \neq \emptyset\}}. 
\end{multline*}
Once we take conditional expectation with respect to $\mathcal{S}$ and apply Cauchy-Schwarz as well as inequality \eqref{elem_ineq_C}, we obtain on the set $\Omega(r,\delta)=\{\sup_{s \in [0,T]} \|\sigma_s-\sigma(r)_s\|\leq \delta\}$
\begin{align*}
\mathbb{E}\big[\big|R_n(r)-R(n)\big| \big| \mathcal{S} \big] \leq K \delta^2 H_n(T).
\end{align*}
We have $H_n(T) \rightarrow H(T)$ for $n \rightarrow \infty$ and $\mathbb{P}(\Omega(r,\delta)) \rightarrow 1$ for $r \rightarrow \infty$ and all $\delta>0$. Hence we get 
\begin{align*}
\lim_{r \rightarrow \infty} \limsup_{n \rightarrow \infty} \mathbb{P}(|R_n(r)-R_n|> \varepsilon)\rightarrow 0.
\end{align*}
\end{proof}

\begin{proposition}\label{clt_prop2} If Condition \ref{cond_centr}(iii) is fulfilled, we have on $\Omega_T^{(d)}$ the $\mathcal{X}$-stable convergence
\begin{small}
\begin{multline*}
n\sum_{t_{i,n}^{(1)} \wedge t_{j,n}^{(2)} \leq T} \big(\Delta_{i,n}^{(1)} N^{(1)}(q) \big)^2\big(\Delta_{j,n}^{(2)} C^{(2)} \big)^2+\big((\Delta_{i,n}^{(1)} C^{(1)} \big)^2\big(\Delta_{j,n}^{(2)} N^{(2)}(q) \big)^2\big)
 \mathds{1}_{\{\mathcal{I}_{i,n}^{(1)} \cap\mathcal{I}_{j,n}^{(2)} \neq \emptyset\}}
\\\overset{\mathcal{L}-s}{\longrightarrow}
\sum_{p:S_p \leq T} \big( \big(\Delta X_{S_p}^1\big)^2\big(\sigma_{S_p}^{(2)}\big)^2 \eta^{(2)}(S_p)
+\big(\Delta X_{S_p}^2\big)^2\big(\sigma_{S_p}^{(1)}\big)^2 \eta^{(1)}(S_p) \big)
\end{multline*}
\end{small}as $n \rightarrow \infty$ and then $q \rightarrow \infty$. Here, $\big(S_p\big)_{p \in \mathbb{N}}$ is an enumeration of the jump times of $X$, and the $\eta^{(l)}(S_p)$ are distributed according to
\begin{align*}
\eta^{(l)}(x) \sim \Gamma^{(l)}(x,dy), \quad l=1,2, \quad x \in [0,T],
\end{align*}
where the $\Gamma^{(l)}$ are defined in Condition \ref{cond_centr}(iii). Furthermore, the $\eta^{(l)}(x)$ are mutually independent and independent of $\mathcal{X}$.
\end{proposition}

\begin{proof}[Proof of Proposition \ref{clt_prop2}]
\textit{Step 1.} Denote by $U_{q,p}$ the jump times of $N(q)$ ordered by the size of $\big\|\int_{\mathbb{R}^2}z \mu(U_{q,p},dz)\big\|$. We begin by showing that Condition \ref{cond_centr}(iii) yields the $\mathcal{X}$-stable convergence of all the $\eta_n^{(l)}(U_{q,p})$ to the respective $\eta^{(l)}(U_{q,p})$, i.e.\ we have to show
\begin{multline}\label{stable_eta}
\mathbb{E} \big[\Lambda f\big(\big(
n \eta_n^{(1)}(U^{(1)}_{q,p}) \big)_{U^{(1)}_{q,p} \leq T},\big(n \eta_n^{(2)}(U^{(2)}_{q,p}) \big)_{U^{(2)}_{q,p} \leq T}
\big)  \big]
\\ \rightarrow
\widetilde{\mathbb{E}} \big[\Lambda f\big(\big(
 \eta^{(1)}(U^{(1)}_{q,p}) \big)_{U^{(1)}_{q,p} \leq T},\big( \eta^{(2)}(U^{(2)}_{q,p}) \big)_{U^{(2)}_{q,p} \leq T} \big)\big]
\end{multline}
for all $\mathcal{X}$-measurable bounded random variables $\Lambda$ and all continuous bounded functions $f$. Here $(U^{(l)}_{q,p})_{p \in \mathbb{N}}$ is an enumeration of the jump times of $N^{(l)}(q)$ where the jumps are ordered again by the size of $\big\|\int_{\mathbb{R}^2}z \mu(U^{(l)}_{q,p},dz)\big\|$. We only give a sketch of the proof here and refer for more details to the proofs of Lemma 5.8 in \cite{Jac08} and Lemma 6.2 in \cite{JacPro98}.

By conditioning on the $\sigma$-algebra $\mathcal{G}$ generated by the Brownian motion $W$, the process $\rho$ and the jump times $S_p \leq T$ we see that it is sufficient to prove \eqref{stable_eta} for any $\Lambda$ of the form
\begin{align}\label{clt_prop2_lambda}
\Lambda=\gamma(W)\xi(\rho) \kappa\big((U^{(1)}_{q,p})_{p \in \mathbb{N}},(U^{(2)}_{q,p})_{p \in \mathbb{N}} \big)
\end{align}
with Borel-measurable functions $\gamma,\xi,\kappa$, because the $\eta_n^{(l)}(S_p)$ depend on $\mathcal{X}$ only through $\mathcal{G}$, and all $\mathcal{G}$-measurable random variables can be approximated by random variables of the form \eqref{clt_prop2_lambda}.

Next we set $U_{q,p,m}^{(l,-)}=\max\{U_{q,p}^{(l)}-1/m,0\}$, $U_{q,p,m}^{(l,+)}=\min\{U_{q,p}^{(l)}+1/m,T\}$, and
\begin{align*}
&B(m) := \bigcup_{l=1,2} \bigcup_{U_{q,p}^{(l)}\leq T} \big[U_{q,p,m}^{(l,-)},U_{q,p,m}^{(l,+)}\big],
\\&W(m)_t=\int_0^t \big(1-\mathds{1}_{B(m)}(s)\big)dW(s).
\end{align*} 
Denote by $\Omega(q,m,n)$ the set on which $\mathcal{M}_n^{(l)}(U_{q,p}^{(l)})<1/m$ for all $U_{q,p}^{(l)}\leq T$ and where two different jumps $U_{q,p_1}^{(l_1)},U_{q,p_2}^{(l_2)}\leq T$ are further apart than $2|\pi_n|_T$. On this set the process $W(m)$ is independent of all the $\eta_n^{(l)}(U_{q,p}^{(l)})$ by independence of $W$ and $\mu$, and $\eta_n^{(l_1)}(U_{q,p_1}^{(l_1)})$ and $\eta_n^{(l_2)}(U_{q,p_2}^{(l_2)})$ are independent for $U_{q,p_1}^{(l_1)} \neq U_{q,p_2}^{(l_2)}$ because the corresponding increments of the Brownian motion do not overlap. In particular, as there are no common jumps,  on $\Omega(q,m,n)$ the common distribution of the $\eta_n^{(l)}(U_{q,p}^{(l)})$ is independent of $\rho$ and $W(m)$. This yields
\begin{align*}
&\mathbb{E} \big[\gamma(W)\xi(\rho) \kappa\big((U^{(1)}_{q,p})_{p \in \mathbb{N}},(U^{(2)}_{q,p})_{p \in \mathbb{N}} \big) f\big(\big(\big(
n \eta_n^{(l)}(U^{(l)}_{q,p}) \big)_{U^{(l)}_{q,p} \leq T}\big)_{l=1,2}\big)   \big]
\\&~~= \lim_{m \rightarrow \infty} \lim_{n \rightarrow \infty} \mathbb{E} \big[\mathds{1}_{\Omega(q,m,n)}\gamma(W(m))\xi(\rho)\big] 
\\&~~~~~~~~~~~\times \mathbb{E} \big[\mathds{1}_{\Omega(q,m,n)}\kappa\big((U^{(1)}_{q,p})_{p \in \mathbb{N}},(U^{(2)}_{q,p})_{p \in \mathbb{N}} \big) f\big(\big(\big(
n \eta_n^{(l)}(U^{(l)}_{q,p}) \big)_{U^{(l)}_{q,p} \leq T}\big)_{l=1,2}\big)  \big]
\end{align*}
because of $\mathbb{P}(\Omega(q,m,n))\rightarrow 1$ as $n \to \infty$, and it is sufficient to show
\begin{multline*}
\mathbb{E} \big[\kappa\big((U^{(1)}_{q,p})_{p \in \mathbb{N}},(U^{(2)}_{q,p})_{p \in \mathbb{N}} \big) f \big(\big(\big(
n \eta_n^{(l)}(U^{(l)}_{q,p}) \big)_{U^{(l)}_{q,p} \leq T}\big)_{l=1,2}\big)   \big]
\\\rightarrow \widetilde{\mathbb{E}} \big[\kappa\big((U^{(1)}_{q,p})_{p \in \mathbb{N}},(U^{(2)}_{q,p})_{p \in \mathbb{N}} \big) f\big(\big(\big(
 \eta^{(l)}(U^{(l)}_{q,p}) \big)_{U^{(l)}_{q,p} \leq T}\big)_{l=1,2} \big)  \big].
\end{multline*}
Using the standard metric on an infinite Cartesian product, this is exactly Condition \ref{cond_centr}(iii) as conditional on the event that there are $k_l$ jumps of $N^{(l)}(q)$, $l=1,2$, in $[0,T]$ all the $U_{q,p}^{(l)}$ are independent uniformly distributed on $[0,T]$. Note that we may consider functions again which factorize over the $n \eta_n^{(l)}(U_{q,p}^{(l)})$. Hence we have shown
\begin{align}\label{stable_eta2}
\big(\big(\big( n \eta_n^{(l)}(U_{q,p}^{(l)})\big)_{U_{q,p}^{(l)}\leq T}\big)_{l =1,2}\big)
\overset{\mathcal{L}-s}{\longrightarrow}
\big(\big(\big(  \eta^{(l)}(U_{q,p}^{(l)})\big)_{U_{q,p}^{(l)}\leq T}\big)_{l =1,2}\big).
\end{align}

 \textit{Step 2.}
We reconsider the discretized functions $\sigma(r)$ and $C(r)$ from the proof of Proposition \ref{clt_prop1}. Denote by $\Omega(q,r,n)$ the set where two different jumps $U_{q,p_1}^{(l_1)} \neq U_{q,p_2}^{(l_2)}$ are further apart than $2|\pi_n|_T$ and the jumps $U_{q,p}^{(l)}$ are further apart than $2|\pi_n|_T$ from the discontinuities $k/2^r$ of $\sigma(r)$. On this set we get
\begin{small}
\begin{align}\label{clt_prop2_eqn1}
&n \sum_{i,j:t_{i,n}^{(l)} \wedge t_{j,n}^{(l)} \leq T} \bigg(\big(\Delta_{i,n}^{(1)} N^{(1)}(q) \big)^2\big(\Delta_{j,n}^{(2)} C^{(2)}(r) \big)^2 \nonumber
\\&~~~~~~~~~~~~~~~~~~~~~~~~~~~~~~~~~~~~~~~~~~~~~~~~~~~ +\big((\Delta_{i,n}^{(1)} C^{(1)}(r) \big)^2\big(\Delta_{j,n}^{(2)} N^{(2)}(q) \big)^2\bigg)\mathds{1}_{\{\mathcal{I}_{i,n}^{(1)} \cap\mathcal{I}_{i,n}^{(2)} \neq \emptyset\}} \nonumber
\\&=\sum_{p:U_{q,p} \leq T}
 \bigg(\big(\Delta N^{(1)}(q)_{U_{q,p}} \big)^2\big(\sigma^{(2)}(r)_{U_{q,p}} \big)^2 n \eta_n^{(2)}(U_{q,p}) \nonumber
 \\&~~~~~~~~~~~~~~~~~~~~~~~~~~~~~~~~~~~~~~~~~~~~~~~~~~~+\big(\Delta N^{(2)}(q)_{U_{q,p}} \big)^2\big(\sigma^{(1)}(r)_{U_{q,p}} \big)^2 n \eta_n^{(1)}(U_{q,p})\bigg).
\end{align}
\end{small}Using Proposition 2.2 in \cite{PodVet10} we get from \eqref{stable_eta2} 
\begin{small}
\begin{multline*}
\big(N(q),\sigma(r),\big( \big(U_{q,p}^{(l)}\big)_{U_{q,p}^{(l)}\leq T}\big)_{l =1,2} ,\big(\big( n \eta_n^{(l)}(U_{q,p}^{(l)})\big)_{U_{q,p}^{(l)}\leq T}\big)_{l =1,2} \big)
\\ \overset{\mathcal{L}-s}{\longrightarrow}
\big(N(q),\sigma(r),\big( \big(U_{q,p}^{(l)}\big)_{U_{q,p}^{(l)}\leq T}\big)_{l =1,2} ,\big(\big( \eta^{(l)}(U_{q,p}^{(l)})\big)_{U_{q,p}^{(l)}\leq T}\big)_{l =1,2} \big)
\end{multline*}
\end{small}which yields, using the continuous mapping theorem,
\begin{small}
\begin{multline}\label{stable_conv}
\sum_{p:U_{q,p} \leq T} \sum_{l=1,2}
 \big(\Delta N^{(3-l)}(q)_{U_{q,p}} \big)^2\big(\sigma^{(l)}(r)_{U_{q,p}} \big)^2 n \eta_n^{(l)}(U_{q,p})
\\ \overset{\mathcal{L}-s}{\longrightarrow}
\sum_{p:U_{q,p} \leq T} \sum_{l=1,2}
 \big(\Delta N^{(3-l)}(q)_{U_{q,p}} \big)^2\big(\sigma^{(l)}(r)_{U_{q,p}} \big)^2 \eta^{(l)}(U_{q,p}).
\end{multline}
\end{small}Note that we may replace the left hand side of \eqref{stable_conv} by one of \eqref{clt_prop2_eqn1}, since $\mathbb{P}(\Omega(q,r,n))\rightarrow 1$ as $n \to \infty$.

But the convergence in \eqref{stable_conv} is even preserved if we replace $\sigma(r)$ by $\sigma$, because we get convergence in probability for both sides as $r \rightarrow \infty$: For the left hand side of \eqref{clt_prop2_eqn1} we use that the number of jumps of $N(q)$ and their size is bounded in probability and a similar argument as for the last step in the proof of Proposition \ref{clt_prop1}. For the right hand side of \eqref{stable_conv} we use in addition that the first moments of the $\eta^{(l)}(s)$ are uniformly bounded.
\\ \textit{Step 3.} We have
\begin{small}
\begin{align}\label{clt_prop2_conv_q}
&\sum_{p:S_{p} \leq T} \sum_{l=1,2}
 \big(\Delta X^{(3-l)}_{S_p} \big)^2\big(\sigma^{(l)}_{S_p} \big)^2 \eta^{(l)}(S_p)
-\sum_{p:U_{q,p} \leq T} \sum_{l=1,2}
 \big(\Delta N^{(3-l)}(q)_{U_p} \big)^2\big(\sigma^{(l)}_{U_p} \big)^2 \eta^{(l)}(U_p) \nonumber
 \\&~~~~~~~~~~~~~~~~=\sum_{p:S_{p} \leq T} \sum_{l=1,2}
 \big(\Delta M^{(3-l)}(q)_{S_p} \big)^2\big(\sigma^{(l)}_{S_p} \big)^2 \eta^{(l)}(S_p).
\end{align}
\end{small}Computing the $\mathcal X$-conditional expectation first and applying dominated convergence afterwards, it is easy to see that the right hand side of \eqref{clt_prop2_conv_q} converges to zero in probability as $q \rightarrow \infty$. This finishes the proof of Proposition \ref{clt_prop2}.
\end{proof}

The following lemma is needed for the proof of Proposition \ref{clt_prop3}.

\begin{lemma} \label{clt_lemma1}
Let Condition \ref{cond_consistency} be satisfied. Then there exists a constant $K$ which is independent of $(i,j)$ such that
\begin{align*} 
\mathbb{E} \big[ \big(\Delta_{i,n}^{(l)} C^{(l)} \big)^2\big(\Delta_{j,n}^{(3-l)} M^{(3-l)}(q) \big)^2\big| \mathcal{S} \big] \leq K e_q \big|\mathcal{I}_{i,n}^{(l)}\big|\big|\mathcal{I}_{j,n}^{(3-l)}\big|, \quad l=1,2. 
\end{align*}
On the set $\Omega_T^{(d)}$ we further have
\begin{align} 
\mathbb{E} \big[ \big(\Delta_{i,n}^{(l)} M^{(l)}(q) \big)^2\big(\Delta_{j,n}^{(3-l)} M^{(3-l)}(q') \big)^2\big| \mathcal{S} \big] \leq K e_q e_{q'} \big|\mathcal{I}_{i,n}^{(1)}\big|\big|\mathcal{I}_{j,n}^{(2)}\big|. \label{abschatzung_klein_klein}
\end{align}
\end{lemma}

\begin{proof}[Proof of Lemma \ref{clt_lemma1}]
If $\mathcal{I}_{i,n}^{(l)} \cap \mathcal{I}_{j,n}^{(3-l)}=\emptyset$ we use iterated expectations and Lemma \ref{elem_ineq}. If the intervals do overlap, there exists a $k$ with $\mathcal{I}_{i,n}^{(l)} \cap \mathcal{I}_{j,n}^{(3-l)}=\big(T_{k-1}^n,T_k^n \big]$. Using iterated expectations we get
\begin{multline*}
\mathbb{E} \big[ \big(\Delta_{i,n}^{(l)} C^{(l)} \big)^2\big(\Delta_{j,n}^{(3-l)} M^{(3-l)}(q) \big)^2\big| \mathcal{S} \big]
\\ \leq  K e_q \big(\big(\Delta_{k-1}^{n,1,-}+\Delta_{k}^{n}+\Delta_{k}^{n,1,+}\big)\big(\Delta_{k-1}^{n,2,-}+\Delta_{k}^{n}+\Delta_{k}^{n,2,+}\big)- \big(\Delta_k^n \big)^2\big)
\\+\mathbb{E} \big[ \big( C^{(l)}_{T_k^n}-C^{(l)}_{T_{k-1}^n} \big)^2\big(M^{(3-l)}(q)_{T_k^n}-M^{(3-l)}(q)_{T_{k-1}^n} \big)^2\big| \mathcal{S} \big]
\end{multline*}
and an analogous result for \eqref{abschatzung_klein_klein}. The claim now follows from Lemma 8.2 in \cite{JacTod09} which is basically Lemma \ref{clt_lemma1} for $\mathcal{I}_{i,n}^{(l)} = \mathcal{I}_{j,n}^{(3-l)}$. The generalization to $q\neq q'$ here does not complicate the proof.
\end{proof}

\begin{proposition} \label{clt_prop3}
If Condition \ref{cond_centr}(i)-(ii) is fulfilled, we have on $\Omega_T^{(d)}$
\begin{align*}
\lim_{q \rightarrow \infty} \limsup_{n \rightarrow \infty} 
\mathbb{P} \big( \big| n V(f,\pi_n)_T - R(n,q)_T \big| > \varepsilon\big)=0 \quad \forall \varepsilon>0
\end{align*}
with
\begin{multline*}
R(n,q)_T=n\sum_{i,j:t_{i,n}^{(1)} \wedge t_{j,n}^{(2)} \leq T} \big(\big(\Delta_{i,n}^{(1)} C^{(1)} \big)^2\big(\Delta_{j,n}^{(2)} C^{(2)} \big)^2
+\big(\Delta_{i,n}^{(1)} N^{(1)}(q) \big)^2\big(\Delta_{j,n}^{(2)} C^{(2)} \big)^2
\\+\big(\Delta_{i,n}^{(1)} C^{(1)} \big)^2\big(\Delta_{j,n}^{(2)} N^{(2)}(q) \big)^2\big)
\mathds{1}_{\{\mathcal{I}_{i,n}^{(1)} \cap\mathcal{I}_{i,n}^{(2)} \neq \emptyset\}}.
\end{multline*}
\end{proposition}

\begin{proof}[Proof of Proposition \ref{clt_prop3}]
Since $\gamma$ is bounded by Condition \ref{cond_consistency} we can write
\begin{align*}
X=X_0+B(q')+C+M(q')
\end{align*}
on $[0,T]$ for some positive number $q'$ (not necessarily an integer) which yields
\begin{align}\label{N_q_qstricht}
N(q)=B(q')-B(q)+M(q')-M(q).
\end{align}
We apply inequality \eqref{algebr_ineq} with $a_l=0,b_l=\Delta_{i,n}^{(l)}B^{(l)}(q),c_l=\Delta_{i,n}^{(l)}M^{(l)}(q)$ and \mbox{$d_l=\Delta_{i,n}^{(l)}C^{(l)}+\Delta_{i,n}^{(l)}N^{(l)}(q)$}. Then, by Propositions \ref{clt_prop1} and \ref{clt_prop2} we have
\begin{small}
\begin{align} \label{clt_prop3_d_l}
&\rho n\sum_{i,j:t_{i,n}^{(1)} \wedge t_{j,n}^{(2)} \leq T} \big(\Delta_{i,n}^{(1)} C^{(1)} + \Delta_{i,n}^{(1)}N^{(1)}(q) \big)^2\big(\Delta_{j,n}^{(2)} C^{(2)}+\Delta_{j,n}^{(2)} N^{(2)}(q) )\big)^2
\mathds{1}_{\{\mathcal{I}_{i,n}^{(1)} \cap\mathcal{I}_{i,n}^{(2)} \neq \emptyset\}}
\\&~~ \leq \rho  n\sum_{i,j:t_{i,n}^{(1)} \wedge t_{j,n}^{(2)} \leq T} 4\big((\Delta_{i,n}^{(1)} C^{(1)})^2 +(\Delta_{i,n}^{(1)} N^{(1)}(q))^2 \big)\big((\Delta_{j,n}^{(2)} C^{(2)})^2+(\Delta_{j,n}^{(2)} N^{(2)}(q) )^2\big) \nonumber \\
&~~~~~~~~~~~~~~~~~~~~~~~~~~~~~~~~~~~~~~~~~~~~~~~~~~~~~~~~~~~~~~~~~~~~~~~~~~~ \times
\mathds{1}_{\{\mathcal{I}_{i,n}^{(1)} \cap\mathcal{I}_{i,n}^{(2)} \neq \emptyset\}}, \nonumber
\end{align}
\end{small}and the latter term is bounded in probability. Hence, it converges to zero for $\rho \rightarrow 0$.

We also get for $l=1,2$ using \eqref{N_q_qstricht}, Lemma \ref{elem_ineq} and Lemma \ref{clt_lemma1},
\begin{align*}
&\mathbb{E} \big[ c_\rho n\sum_{i,j:t_{i,n}^{(3-l)} \wedge t_{j,n}^{(l)} \leq T} \big(\big(\Delta_{i,n}^{(3-l)} B^{(3-l)}(q)\big)^2 +\big(\Delta_{i,n}^{(3-l)} M^{(3-l)}(q)\big)^2 \big)
\\ &~~\times \big(\big(\Delta_{j,n}^{(l)} B^{(l)}(q)\big)^2+\big(\Delta_{j,n}^{(l)} M^{(l)}(q)\big)^2+
\big(\Delta_{j,n}^{(l)} C^{(l)} +\Delta_{j,n}^{(l)} N^{(l)}(q) \big)^2\big) \\ 
&~~~~~~~~~~~~~~~~~~~~~~~~~~~~~~~~~~~~~~~~~~~~~~~~~~~~~~~~~~~~~~~~~~~~~~~~~\times
\mathds{1}_{\{\mathcal{I}_{i,n}^{(3-l)} \cap\mathcal{I}_{i,n}^{(l)} \neq \emptyset\}} \big| \mathcal{S} \big]
\\&~~ \leq  c_\rho n\sum_{i,j:t_{i,n}^{(3-l)} \wedge t_{j,n}^{(l)} \leq T} \big(K_q \big|\mathcal{I}_{i,n}^{(3-l)}\big| +K e_q \big)\big|\mathcal{I}_{i,n}^{(3-l)}\big|
\\ &~~~~\times \big(K_q\big|\mathcal{I}_{j,n}^{(l)}\big|+K e_q +2K  + 8(K_q+K_{q'}) \big|\mathcal{I}_{j,n}^{(l)}\big|+8K (e_q + e_{q'})
\big)\big|\mathcal{I}_{j,n}^{(l)}\big| \\
&~~~~~~~~~~~~~~~~~~~~~~~~~~~~~~~~~~~~~~~~~~~~~~~~~~~~~~~~~~~~~~~~~~~~~~~~~~~ \times
\mathds{1}_{\{\mathcal{I}_{i,n}^{(3-l)} \cap\mathcal{I}_{i,n}^{(l)} \neq \emptyset\}}
\\&~~\leq c_\rho \big(K_q |\pi_n|_T + K e_q   \big)
\big(H_n(T)+O_{\mathbb{P}}\big(n (|\pi_n|_T)^2 \big) \big),
\end{align*}
where the latter bound converges to zero for $n \rightarrow \infty$ and then $q \rightarrow \infty$. Therefore, inequality \eqref{algebr_ineq} shows that only the terms as in \eqref{clt_prop3_d_l} remain in the limit. On $\Omega_T^{(d)}$, the terms that occur in \eqref{clt_prop3_d_l} but not in $R(n,q)_T$ are of the form
\begin{multline}\label{clt_prop3_remain}
n\sum_{i,j:t_{i,n}^{(l)} \wedge t_{j,n}^{(3-l)} \leq T} \big(\big(\Delta_{i,n}^{(l)} C^{(l)}\big)^2 +2 \big(\Delta_{i,n}^{(l)} C^{(l)}\big)\big(\Delta_{i,n}^{(l)}N^{(l)}(q)\big)+\big(\Delta_{i,n}^{(l)}N^{(l)}(q) \big)^2 \big)
\\ \times
\big(\big(\Delta_{j,n}^{(3-l)} C^{(3-l)}\big)\big(\Delta_{j,n}^{(3-l)} N^{(3-l)}(q) \big)\big)
\mathds{1}_{\{\mathcal{I}_{i,n}^{(l)} \cap\mathcal{I}_{j,n}^{(3-l)} \neq \emptyset\}}, \quad l=1,2.
\end{multline}
From similar arguments as before, we obtain that the sum over terms containing the product $\big(\Delta_{i,n}^{(l)}N^{(l)}(q)\big)\big(\Delta_{j,n}^{(3-l)} N^{(3-l)}(q) \big)$ converges to zero because we are on $\Omega_T^{(d)}$. 

In order to discuss why the remaining terms vanish, let $\Omega(3-l,q,\Gamma)$ be the set on which $N^{(3-l)}(q)$ has no more than $\Gamma$ jumps in $[0,T]$. On $\Omega(3-l,q,\Gamma)$ we get, using the Cauchy-Schwarz inequality and \eqref{elem_ineq_C},
\begin{align*}
&\mathbb{E} \big[n\sum_{i,j:t_{i,n}^{(l)} \wedge t_{j,n}^{(3-l)} \leq T} \big(\Delta_{i,n}^{(l)} C^{(l)}\big)^2 
\big(\Delta_{j,n}^{(3-l)} C^{(3-l)}\big)\big(\Delta_{j,n}^{(3-l)} N^{(3-l)}(q) )\big)
\mathds{1}_{\{\mathcal{I}_{i,n}^{(l)} \cap\mathcal{I}_{i,n}^{(3-l)} \neq \emptyset\}} \big| \mathcal{S} \big]
\\&~~~~~\leq 3 K \Gamma n \big(|\pi_n|_T\big)^{3/2}
\end{align*}
which tends to zero by Condition \ref{cond_centr}. Because of $\mathbb{P}(\Omega(3-l,q,\Gamma))\rightarrow 1$ for $\Gamma \rightarrow \infty$ we find that all the terms in  \eqref{clt_prop3_remain} vanish as $n \rightarrow \infty$.
\end{proof}

\begin{proof}[Proof of Theorem \ref{clt}]
This is a direct consequence of Propositions \ref{clt_prop1}, \ref{clt_prop2} and \ref{clt_prop3} as well as \eqref{cons_Vg}.
\end{proof}

\subsection{Proof for the testing procedure}
\bibliographystyle{chicago}

\begin{proof}[Proof of Theorem \ref{test_theo}]
We will only show
\begin{align}\label{test_theo_level2}
\widetilde{\mathbb{P}}\big(nV(f,\pi_n)_T > A_{n,T}(\beta, \varpi)+Q_{n,T}(1-\alpha) \big| F^{(d)} \big) \rightarrow \alpha,
\end{align}
which is well-defined on the entire $\Omega_{\mathcal X}$. To this end, we will prove in the sequel that Condition \ref{cond_centr} ensures
\begin{align}\label{estimator_cont}
A_{n,T}(\beta,\varpi) \overset{\mathbb{P}}{\longrightarrow} \widetilde{C}_T,
\end{align}
while under Condition \ref{condition_test} both $Q_{n,T}(\alpha)=\widehat{Q}_{n,T}(\alpha)$ or $Q_{n,T}(\alpha)=\widetilde{Q}_{n,T}(\alpha)$ satisfy
\begin{align}\label{estimator_jump}
Q_{n,T} (\alpha) \overset{\widetilde{\mathbb{P}}}{\longrightarrow} Q(\alpha)
\end{align}
for each $\alpha \in [0,1]$, where $Q(\alpha)$ denotes the \mbox{$\mathcal{X}$-conditional} $\alpha$ quantile of $\widetilde{D}_T$, i.e.\ the $\mathcal{X}$-measurable random variable defined via 
\[
\widetilde{\mathbb{P}}\big(\widetilde{D}_T \leq Q(\alpha) \big| \mathcal{X} \big)=\alpha.
\]
Note that the $\cal X$-conditional distribution of $\widetilde{D}_T$ is continuous by Condition \ref{cond_centr} (iii). Then, by Theorem \ref{clt} and \eqref{estimator_cont} we get
\begin{align*}
n V(f,\pi_n)_T -A_{n,T}(\beta,\varpi) \overset{\mathcal{L}-s}{\longrightarrow} \widetilde{D}_T
\end{align*}
which yields 
\begin{align*}
&\mathbb{P} \big(\big\{n V(f,\pi_n)_T >A_{n,T}(\beta,\varpi)+Q_{n,T}(1-\alpha) \big\} \cap F^{(d)}  \big) 
\\&~~= \mathbb{P} \big(\big\{n V(f,\pi_n)_T -A_{n,T}(\beta,\varpi)>Q_{n,T}(1-\alpha) \big\}  \cap F^{(d)}  \big) \\&~~\rightarrow \widetilde{\mathbb{P}}\big(\big\{\widetilde{D}_T>Q(1-\alpha) \big\} \cap F^{(d)}  \big)=\alpha \mathbb{P}\big(F^{(d)}\big)
\end{align*}
by \eqref{estimator_jump} and the definition of $Q(\alpha)$. This implies \eqref{test_theo_level2} and hence \eqref{test_theo_level}.

The consistency claim \eqref{test_theo_power} follows from the fact that $\widetilde{\Phi}_{n,T}^{(d)}$ converges to a strictly positive limit on $\Omega_T^{(j)}$ while $\mathpzc{c}_n=O_{\widetilde{\mathbb{P}}}(n^{-1})$.
\end{proof}

\begin{proof}[Proof of \eqref{estimator_cont}]
Following the proof of Proposition \ref{clt_prop3}, it is enough to show that
\begin{small}
\begin{multline*}
n \sum_{i,j:t_{i,n}^{(1)} \wedge t_{j,n}^{(2)} \leq T} \bigg(\big(\Delta_{i,n}^{(1)} C^{(1)}\big)^2 \big(\Delta_{j,n}^{(2)} C^{(2)}\big)^2
+\big(\Delta_{i,n}^{(1)} N^{(1)}(q)\big)^2 \big(\Delta_{j,n}^{(2)} C^{(2)}\big)^2
\\+\big(\Delta_{i,n}^{(1)} C^{(1)}\big)^2 \big(\Delta_{j,n}^{(2)} N^{(2)}(q)\big)^2
 \bigg) 
\times
\mathds{1}_{\{|\Delta_{i,n}^{(1)} X^{(1)} |\leq \beta |\mathcal{I}_{i,n}^{(1)}|^{\varpi} \wedge |\Delta_{j,n}^{(2)} X^{(2)} |\leq \beta |\mathcal{I}_{j,n}^{(2)}|^{\varpi}\}}
 \mathds{1}_{\{\mathcal{I}_{i,n}^{(1)} \cap \mathcal{I}_{j,n}^{(2)} \neq \emptyset\}}
\end{multline*}
\end{small}converges to $\widetilde{C}_T$. 

We first deal with the cross terms of big jumps and Brownian increments. Let $\Omega(q,\Gamma,\delta,n)$ denote the set on which there are at most $\Gamma$ jumps of $N(q)$ which are of size $|\Delta N(q)|>\delta$ and which are further apart than $2|\pi_n|_T$. If $N(q)^{(l)}$ jumps in $\mathcal{I}_{i,n}^{(l)}$, we have on the set $\Omega(q,\Gamma,\delta,n)$
\begin{align*}
&\mathbb{E}\big[\big(\Delta_{i,n}^{(l)}X^{(l)}\big)^2\big| \mathcal{S} \big] 
\\ &~~ \geq \delta^2 - \mathbb{E}\big[\big|\big(\Delta_{i,n}^{(l)}X^{(l)}\big)^2-\big(\Delta_{i,n}^{(l)}N^{(l)}(q)\big)^2\big|\big| \mathcal{S} \big] 
\\ &~~\geq 
\delta^2-3 \mathbb{E}\big[\big(\Delta_{i,n}^{(l)}B^{(l)}(q)\big)^2+\big(\Delta_{i,n}^{(l)}C^{(l)}\big)^2+\big(\Delta_{i,n}^{(l)}M^{(l)}(q)\big)^2| \mathcal{S} \big]
\\&~~~~~~~~~~~~~~~~~~~~~~~~-2  \mathbb{E}\big[\big|\Delta_{i,n}^{(l)}B^{(l)}(q)+\Delta_{i,n}^{(l)}C^{(l)}+\Delta_{i,n}^{(l)}M^{(l)}(q)\big| \big|\Delta_{i,n}^{(l)}N^{(l)}(q) \big| \big| \mathcal{S} \big]
\\ &~~ \geq \delta^2 - 3 K_q |\pi_n|_T-2\Delta K_q \sqrt{|\pi_n|_T}
\end{align*}
where we used Cauchy-Schwarz inequality, Lemma \ref{elem_ineq} and the fact that $\gamma$ is bounded by a constant $\Delta$. This lower bound converges to $\delta^2>0$ as $n \rightarrow \infty$. Hence $\mathbb{P}(|\Delta_{i_n,n}^{(l)} X^{(l)} |\leq \beta |\mathcal{I}_{i_n,n}^{(l)}|^\varpi)$ converges to zero where we use $i_n$ to indicate in which interval the jump is observed. Because of $\mathbb{P}(\Omega(q,\Gamma,\delta,n))\rightarrow 1$ for as first $n \to \infty$, then $\Gamma \rightarrow \infty$ and finally $\delta \rightarrow 0$, the sum over the cross terms therefore vanishes.  

Using Proposition \ref{clt_prop1} it remains to show 
\begin{multline}\label{conv_Lt}
\widetilde{L}_T=n \sum_{i,j:t_{i,n}^{(1)} \wedge t_{j,n}^{(2)} \leq T} \big(\Delta_{i,n}^{(1)} C^{(1)}\big)^2
\big(\Delta_{j,n}^{(2)} C^{(2)}\big)^2
\\ \times
\mathds{1}_{\{|\Delta_{i,n}^{(1)} X^{(1)} > \beta |\mathcal{I}_{i,n}^{(1)}|^{\varpi} \vee |\Delta_{j,n}^{(2)} X^{(2)} |> \beta |\mathcal{I}_{j,n}^{(2)}|^{\varpi}\}}
 \mathds{1}_{\{\mathcal{I}_{i,n}^{(1)} \cap \mathcal{I}_{j,n}^{(2)} \neq \emptyset\}} \overset{\mathbb{P}}{\longrightarrow} 0.
\end{multline}
The conditional Markov inequality plus an application of Lemma \ref{elem_ineq} give
\begin{align}\label{markov}
\mathbb{P} \big(\big|\Delta_{i,n}^{(l)}X^{(l)}\big|> \beta \big|\mathcal{I}_{i,n}^{(l)}\big|^\varpi \big| \mathcal{S}\big) \le K \big|\mathcal{I}_{i,n}^{(l)}\big|^{1-2 \varpi}.
\end{align}
Using
\begin{multline*}
\widetilde{L}_T \leq n \sum_{i_1,i_2:t_{i_1,n}^{(1)} \wedge t_{i_2,n}^{(2)} \leq T} \big(\Delta_{i_1,n}^{(1)} C^{(1)}\big)^2
\big(\Delta_{i_2,n}^{(2)} C^{(2)}\big)^2
\\ \times
\sum_{l=1,2}
\mathds{1}_{\{|\Delta_{i_l,n}^{(l)} X^{(l)} > \beta |\mathcal{I}_{i_l,n}^{(l)}|^{\varpi} \}}
\mathds{1}_{\{\mathcal{I}_{i_1,n}^{(1)} \cap \mathcal{I}_{i_2,n}^{(2)} \neq \emptyset\}}
\end{multline*}
and the generalized H\"older inequality, as well as Lemma \ref{elem_ineq_C} and \eqref{markov}, we get
\begin{align} \label{Hoelder}
\mathbb{E}\big[\widetilde{L}_T \big| \mathcal{S} \big]
&\leq K n \sum_{i_1,i_2:t_{i_1,n}^{(1)} \wedge t_{i_2,n}^{(2)} \leq T} \big|\mathcal{I}_{i_1,n}^{(1)} \big|\big|\mathcal{I}_{i_2,n}^{(2)} \big| \sum_{l=1,2}
\big|\mathcal{I}_{i_l,n}^{(l)} \big|^{(1-2 \varpi)/p'}\mathds{1}_{\{\mathcal{I}_{i,n}^{(1)} \cap \mathcal{I}_{j,n}^{(2)} \neq \emptyset\}} \nonumber
\\& \leq K (|\pi_n|_T)^{(1-2 \varpi)/p'} H_n(T)
\end{align}
for any $p'>1$, which tends to zero by Condition \ref{cond_consistency} and Condition \ref{cond_centr}(ii). This yields \eqref{conv_Lt}.
\end{proof}

For the proof of \eqref{estimator_jump} we need a few preliminary results which yield that the convergence of the empirical $\mathcal{X}$-conditional distribution of the $\widehat{D}_{T,n,m}$ and $\widetilde{D}_{T,n,m}$ to the $\mathcal{X}$-conditional distribution of $\widetilde{D}_T$ follows from the convergence of the common empirical distribution of the $\hat{\eta}^{(l_j)}_{n,m}(s_j)$ to the common distribution of the $\eta^{(l_j)}(s_j)$ provided in Condition \ref{condition_test}. These results are proved in Lemma \ref{emp_lemma} and Proposition \ref{emp_prop}.

\begin{lemma}\label{emp_lemma} 
Suppose that $A_{n,j} \overset{\widetilde{\mathbb{P}}}{\longrightarrow} A_j$ for $\mathcal{X}$-measurable $A_j > 0$, and let $S_j \in [0,T]$, $j=1,\ldots,J,$ be almost surely distinct $\mathcal{X}$-measurable random variables as well. Then, under Condition \ref{condition_test},
\begin{align*}
\widetilde{\mathbb{P}}\Bigg(\Bigg|\frac{1}{M_n} \sum_{m=1}^{M_n} \mathds{1}_{\{\sum_{j=1}^{J} A_{n,j} \hat{\eta}_{n,m}^{(l_j)}(S_j)\leq Z\}} - \widetilde{\mathbb{P}}\bigg(\sum_{j=1}^{J} A_{j} \eta^{(l_j)}(S_j) \leq Z \bigg| \mathcal{X} \bigg) \Bigg|> \varepsilon \Bigg) \rightarrow 0
\end{align*}
for any $\mathcal{X}$-measurable random variable $Z$ and any $\varepsilon>0$. 
\end{lemma}

\begin{proof}

First, note that
\begin{align}
& \widetilde{\mathbb{P}}\big(\big|\frac{1}{M_n}  \sum_{m=1}^{M_n} \big(\mathds{1}_{\{\hat{\eta}^{(l_j)}_{n,m}(s_j)\leq x_j,~j=1,\ldots,J\}} 
- \widetilde{\mathbb{P}}(\eta^{(l_j)}(s_j) \leq x_j,~j=1,\ldots,J)\big) \big|> \varepsilon \big) \label{emp_distr1}\\
&~~~~\le
\widetilde{\mathbb{P}}\big(\big|\frac{1}{M_n} \sum_{m=1}^{M_n} \mathds{1}_{\{\hat{\eta}^{(l_j)}_{n,m}(s_j)\leq x_j,~j=1,\ldots,J\}} 
-
 \widetilde{\mathbb{P}}\big(\hat{\eta}^{(l_j)}_{n,1}(s_j)\leq x_j,~j=1,\ldots,J \big| \mathcal S \big) \big|> \frac{\varepsilon}{2} \big) \nonumber
\\&~~~~ +
\widetilde{\mathbb{P}}\big( \widetilde{\mathbb{P}}\big(\hat{\eta}^{(l_j)}_{n,1}(s_j)\leq x_j,~j=1,\ldots,J \big| \mathcal S \big)
-
 \widetilde{\mathbb{P}}\big({\eta}^{(l_j)}(s_j)\leq x_j,~j=1,\ldots,J\big) \big|> \frac{\varepsilon}{2} \big). \nonumber
\end{align}
Conditionally on $\mathcal S$, the $(\hat{\eta}^{(l_j)}_{n,m}(s_j))_j$ are independent and identically distributed as $m$ varies. Therefore, $M_n \to \infty$, the conditional Markov inequality, and dominated convergence ensure that the first term vanishes asymptotically, whereas the second one converges to zero using (\ref{vert_kon1}). We can therefore assume that the term \eqref{emp_distr1} converges to zero.

For the sake of brevity we restrict ourselves in the sequel to the case $J=2$ only, and we denote by $\Omega(n,\delta)$ the set where $|A_{n,j}-A_j|<\delta$ and $A_{n,j}-\delta>0$ for $j=1,2$. On this set we have $R_n \leq R_n^+(\delta)$ with
\begin{align*}
R_n&=\frac{1}{M_n}  \sum_{m=1}^{M_n} \mathds{1}_{\{ A_{n,1} \hat{\eta}_{n,m}^{(l_1)}(S_1)+A_{n,2} \hat{\eta}_{n,m}^{(l_2)}(S_2)\leq Z\}}, 
\\R_n^+(\delta)&= \frac{1}{M_n} \sum_{m=1}^{M_n} \mathds{1}_{\{ (A_{1}-\delta) \hat{\eta}_{n,m}^{(l_1)}(S_1)+(A_{2}-\delta) \hat{\eta}_{n,m}^{(l_2)}(S_2)\leq Z\}}. 
\end{align*}
Using a discretization we get that $R_n^+(\delta)$ is bounded by
\begin{align*}
R_n^+(\delta,r)=\sum_{z \in \mathbb{Z}} \frac{1}{M_n}  \sum_{m=1}^{M_n} \mathds{1}_{\{ (A_{1}-\delta) \hat{\eta}_{n,m}^{(l_1)}(S_1)\leq Z- z 2^{-r}\}}\mathds{1}_{\{ z2^{-r}<(A_{2}-\delta) \hat{\eta}_{n,m}^{(l_2)}(S_2)\leq (z+1)2^{-r}\}} 
\end{align*}
for any $r \in \mathbb{N}$. We will work conditionally on $\cal X$ first, which yields that we can treat $A_j,S_j,Z$ as constants in the following.

To this end, let $\widetilde{\mathbb{P}}_{\omega}(\, \cdot \,)$ denote a regular version of the conditional probability $\widetilde{\mathbb{P}}(\, \cdot \,|\mathcal{X})(\omega)$, and we work pointwise in $\omega \in \Omega(n,\delta)$. Using that \eqref{emp_distr1} vanishes asymptotically, for any $\xi > 0$ we can find $K \in \N$ and $N \in \N$ such that 
\[
\widetilde{\mathbb{P}}_{\omega} \Big(  \frac{1}{M_n}  \sum_{m=1}^{M_n} \mathds{1}_{\{ K2^{-r}<(A_{2}-\delta) \hat{\eta}_{n,m}^{(l_2)}(S_2)\}} > \xi \Big) < \xi
\]
for any $n \ge N$ as well as
\[
\widetilde{\mathbb{P}}_{\omega} \Big( K2^{-r}<(A_{2}-\delta) {\eta}^{(l_2)}(S_2)\Big) < \xi.
\]
Now, note that 
\begin{align*}
&\sum_{|z| \leq K} \frac{1}{M_n}  \sum_{m=1}^{M_n} \mathds{1}_{\{ (A_{1}-\delta) \hat{\eta}_{n,m}^{(l_1)}(S_1)\leq Z- z 2^{-r}\}}\mathds{1}_{\{ z2^{-r}<(A_{2}-\delta) \hat{\eta}_{n,m}^{(l_2)}(S_2)\leq (z+1)2^{-r}\}}
\\& \leq R_n^+(\delta,r) \leq \frac{1}{M_n}  \sum_{m=1}^{M_n} \Big( \sum_{|z| \leq K}  \mathds{1}_{\{ (A_{1}-\delta) \hat{\eta}_{n,m}^{(l_1)}(S_1)\leq Z- z 2^{-r}\}}\mathds{1}_{\{ z2^{-r}<(A_{2}-\delta) \hat{\eta}_{n,m}^{(l_2)}(S_2)\leq (z+1)2^{-r}\}}
\\&~~~~~~~~~~~~~~~~~~~~~~~~~~~~~~~~~~~~~~~~~~~~~~~~ + \mathds{1}_{\{ K2^{-r}<|(A_{2}-\delta) \hat{\eta}_{n,m}^{(l_2)}(S_2)|\}} \Big).
\end{align*}
Setting 
\begin{multline*}
R^+(\delta,r)=\sum_{z \in \mathbb{Z}} \widetilde{\mathbb{P}}_\omega \big(  \eta^{(l_1)}(S_1)\leq (Z- z 2^{-r})/(A_1-\delta)\big)
\\ \times \widetilde{\mathbb{P}}_\omega\big(
z2^{-r}/(A_2-\delta)< \eta^{(l_2)}(S_2)\leq (z+1)2^{-r}(A_2-\delta)\big)
\end{multline*}
the previous discussion and another application of the fact that \eqref{emp_distr1} converges to zero show  
\begin{align}\label{emp_dist11}
\widetilde{\mathbb{P}}_\omega(|R_n^+(\delta,r)-R(\delta,r)|>\varepsilon)\rightarrow 0
\end{align}
for all $\varepsilon>0$ and any fixed $r$. Also, as the $\eta^{(l)}$ have densities by Condition \ref{cond_centr}(iii), $R^+(\delta,r)$ converges as $r \rightarrow \infty$ to 
\begin{align*}
R^+(\delta)&=\int \widetilde{\mathbb{P}}_\omega\left((A_1-\delta)\eta^{(l_1)}(S_1)+z \leq Z  \right) d\widetilde{\mathbb{P}}_\omega\left( (A_2-\delta)\eta^{(l_2)}(S_2) \leq z \right)
\\ &=\widetilde{\mathbb{P}}_\omega\big( (A_1-\delta) \eta^{(l_1)}(S_1)+(A_2-\delta)\eta^{(l_2)}(S_2)\leq Z  \big),
\end{align*}
using a Riemann sum argument. Together with \eqref{emp_dist11} we obtain
\begin{align}\label{emp_dist12}
&\limsup_{r \rightarrow \infty} \lim_{n \rightarrow \infty}
\widetilde{\mathbb{P}}_\omega (|R_n^+(\delta,r)-R^+(\delta)|>\varepsilon) \nonumber
\\&~~~~ \leq \limsup_{r \rightarrow \infty} \lim_{n \rightarrow \infty}
\widetilde{\mathbb{P}}_\omega (|R_n^+(\delta,r)-R^+(\delta,r)|>\varepsilon/2) \nonumber
\\&~~~~~~~+ \limsup_{r \rightarrow \infty} 
\widetilde{\mathbb{P}}_\omega (|R^+(\delta,r)-R^+(\delta)|>\varepsilon/2),
\end{align}
and the right hand side vanishes. Analogously, if
\begin{align*}
R^-(\delta)=\widetilde{\mathbb{P}}_\omega\big( (A_1+\delta) \eta^{(l_1)}(S_1)+(A_2+\delta)\eta^{(l_2)}(S_2)\leq Z  \big)
\end{align*}
we obtain similar lower bounds $R_n^-(\delta)$ and $R_n^-(\delta,r)$ of $R_n$ for which
\begin{align}\label{emp_dist13}
\limsup_{r \rightarrow \infty} \lim_{n \rightarrow \infty}
\widetilde{\mathbb{P}}_\omega (|R_n^-(\delta,r)-R^-(\delta)|>\varepsilon)=0.
\end{align}

Finally, let $R=R(0)$. Then, using $R^-_n(\delta,r) \mathds{1}_{\Omega(n,\delta)}\leq R_n \leq R_n^+(\delta,r)\mathds{1}_{\Omega(n,\delta)}$, we have
\begin{align*}
&\lim_{n \rightarrow \infty} \widetilde{\mathbb{P}} (|R_n-R|>\varepsilon) 
\\&~~~~ \leq \limsup_{\delta \rightarrow 0}\limsup_{r \rightarrow 0}\lim_{n \rightarrow \infty} \widetilde{\mathbb{P}} (\{\max\{|R_n^+(\delta,r)-R|,|R_n^-(\delta,r)-R|\}>\varepsilon\} \cap \Omega(n,\delta))
\\&~~~~~~~~+\limsup_{\delta \rightarrow 0}\limsup_{r \rightarrow 0}\lim_{n \rightarrow \infty} (1-\mathbb{P}(\Omega(n,q)))
\\&~~~~\leq \limsup_{\delta \rightarrow 0}\limsup_{r \rightarrow 0}\lim_{n \rightarrow \infty} \Big(\widetilde{\mathbb{P}} (\{|R_n^+(\delta,r)-R^+(\delta)|>\varepsilon/2\} \cap \Omega(n,\delta)) \\ &~~~~~~~~~~~~~~~~~~~~~~~~~~~~~~~~~~~~~~+\widetilde{\mathbb{P}} (\{|R^+(\delta)-R|>\varepsilon/2\} \cap \Omega(n,\delta))  \Big)
\\&~~~~~~~~+\limsup_{\delta \rightarrow 0}\limsup_{r \rightarrow 0}\lim_{n \rightarrow \infty} \Big(\widetilde{\mathbb{P}} (\{|R_n^-(\delta,r)-R^-(\delta)|>\varepsilon/2\}  \cap \Omega(n,\delta))   \\ &~~~~~~~~~~~~~~~~~~~~~~~~~~~~~~~~~~~~~~+\widetilde{\mathbb{P}} (\{|R^-(\delta)-R|>\varepsilon/2\}  \cap \Omega(n,\delta))  \Big)
\\&~~~~~~~~+\limsup_{\delta \rightarrow 0}\limsup_{r \rightarrow 0}\lim_{n \rightarrow \infty} (1-\mathbb{P}(\Omega(n,q))).
\end{align*}
This final bound vanishes because dominated convergence allows to deduce \eqref{emp_dist12} and \eqref{emp_dist13} if we replace $\widetilde{\mathbb{P}}_\omega$ by $\widetilde{\mathbb{P}}$, because $R^-(\delta)$ and $R^+(\delta)$ converge $\widetilde{\mathbb{P}}$-almost surely to $R$ for $\delta \rightarrow 0$ as the $\eta^{(l)}$ admit $\mathcal{X}$-conditional densities, and because of $\widetilde{\mathbb{P}}(\Omega(n,\delta))\rightarrow 1$ as $n \rightarrow \infty$ for all $\delta>0$.
\end{proof}

\begin{proposition}\label{emp_prop}
Suppose that Condition \ref{condition_test} is satisfied. Then, 
\begin{align}\label{emp_distr_hatD}
\widetilde{\mathbb{P}}\big(\big\{\big|\frac{1}{M_n} \sum_{m=1}^{M_n} \mathds{1}_{\{ \widehat{D}_{T,n,m}\leq Z\}} - \widetilde{\mathbb{P}}\big(\widetilde{D}_T \leq Z \big| \mathcal{X} \big) \big|> \varepsilon \big\} \cap\Omega_T^{(d)} \big) \rightarrow 0
\end{align}
for any $\mathcal{X}$-measurable random variable $Z$ and all $\varepsilon>0$. The analogous result holds if we replace $\widehat{D}_{T,n,m}$ by $\widetilde{D}_{T,n,m}$.
\end{proposition}

\begin{proof} Without loss of generality we will prove the result for $\widehat{D}_{T,n,m}$ only.
\\ \textit{Step 1.} Denote by $S_j$, $j=1,\ldots,J$, the jump times of the $J$ largest jumps of $X$ in $[0,T]$ with respect to a fixed norm on $\mathbb R$. Recall that on $\Omega_T^{(d)}$ only one component $X^{(l_j)}$ of $X$ jumps at $S_j$. Therefore, setting
\begin{align*}
A_{n,j}=\big(\widehat{\Delta}_n X^{(l_j)}(S_j)\big)^2 \big(\hat{\sigma}_n^{(3-l_j)}(S_j)\big)^2
\quad \text{and} \quad A_j=\big(\Delta X^{(l_j)}_{S_j}\big)^2 \big(\sigma^{(3-l_j)}_{S_j}\big)^2,
\end{align*}
Lemma \ref{emp_lemma} proves
\begin{align}\label{emp_prop1}
\widetilde{\mathbb{P}}\big(\big|\frac{1}{M_n} \sum_{m=1}^{M_n} \mathds{1}_{\{R(J,n,m)\leq Z\}} -
\widetilde{\mathbb{P}}\big(R(J) \leq Z \big| \mathcal{X} \big) \big|> \varepsilon \big) \rightarrow 0
\end{align}
where we used the notation
\begin{align*}
R(J,n,m)&=\sum_{j=1}^J \big(\widehat{\Delta}_n X^{(l_j)}(S_j)\big)^2 \big(\hat{\sigma}_n^{(3-l_j)}(S_j)\big)^2 \hat{\eta}_{n,m}^{(3-l_j)}(S_j),
\\ R(J)&=\sum_{j=1}^J \big(\Delta X^{(l_j)}_{S_j}\big)^2 \big(\sigma^{(3-l_j)}_{S_j}\big)^2 \eta^{(3-l_j)}(S_j).
\end{align*}
\\ \textit{Step 2.} We prove
\begin{align}\label{emp_prop2}
\lim_{J \rightarrow \infty} \limsup_{n \rightarrow \infty} \frac{1}{M_n} \sum_{m=1}^{M_n} \widetilde{\mathbb{P}} \big( \big|R(J,n,m)-\widehat{D}_{T,n,m} \big|>\varepsilon \big) \overset{\widetilde{\mathbb{P}}}{\longrightarrow} 0.
\end{align}
for all $\varepsilon>0$. Denote by $\Omega(q,J,n)$ the set on which the jumps of $N(q)$ are among the $J$ largest jumps and two different jumps of $N(q)$ are further apart than $|\pi_n|_T$. Obviously, $\mathbb{P}(\Omega(q,J,n))\rightarrow 1$ for $J,n \rightarrow \infty$ for any $q>0$. On the set $\Omega(q,J,n)$ we have 
\begin{align}\label{emp_prop2a}
&\big|R(J,n,m)-\widehat{D}_{T,n,m}  \big| \nonumber
\\&~~ = \sum_{l=1,2} \sum_{t_{i,n}^{(l)}\leq T, t_{i,n}^{(l)}\neq S_j} \big(\Delta_{i,n}^{(l)} B^{(l)}(q)+\Delta_{i,n}^{(l)} C^{(l)}+\Delta_{i,n}^{(l)} M^{(l)}(q)\big)^2\mathds{1}_{\{|\Delta_{i,n}^{(l)}X^{(l)}|>\beta |\mathcal{I}_{i,n}^{(l)}|^\varpi\}} \nonumber
\\&~~~~~~~~~~~~~~~~~~~~~~~~~~~~~~~~~~~~~~~~ \times  \big(\hat{\sigma}_n^{(3-l)}(t_{i,n}^{(l)}) \big)^2 \hat{\eta}_{n,m}^{(3-l)}(t_{i,n}^{(l)}) \nonumber
\\&~~\leq \sum_{l=1,2} \sum_{t_{i,n}^{(l)}\leq T} \big(\Delta_{i,n}^{(l)} B^{(l)}(q)+\Delta_{i,n}^{(l)} C^{(l)}+\Delta_{i,n}^{(l)} M^{(l)}(q)\big)^2 \mathds{1}_{\{|\Delta_{i,n}^{(l)}X^{(l)}|>\beta |\mathcal{I}_{i,n}^{(l)}|^\varpi\}} \nonumber
\\& ~~~~~~~~~~~~~~\times \frac{1}{2b_n} \sum_{j:|t_{j,n}^{(3-l)}-t_{i,n}^{(l)}|\leq b_n} \big( \Delta_{j,n}^{(3-l)}X^{(3-l)}\big)^2 \hat{\eta}_{n,m}^{(3-l)}(t_{i,n}^{(l)}) . 
\end{align}
We first consider the increments over the overlapping observation intervals in the right hand side of \eqref{emp_prop2a}. The $\mathcal{F}$-conditional mean of their sum is bounded by
\begin{multline}\label{emp_prop2b}
\frac{3|\pi_n|_T}{2b_n} n \sum_{l=1,2} \sum_{t_{i,n}^{(l)},t_{j,n}^{(3-l)}\leq T} \big(\Delta_{i,n}^{(l)} B^{(l)}(q)+\Delta_{i,n}^{(l)} C^{(l)}+\Delta_{i,n}^{(l)} M^{(l)}(q)\big)^2 \mathds{1}_{\{|\Delta_{i,n}^{(l)}X^{(l)}|>\beta |\mathcal{I}_{i,n}^{(l)}|^\varpi\}}
\\ \times \big( \Delta_{j,n}^{(3-l)}X^{(3-l)}\big)^2 
\mathds{1}_{\{\mathcal{I}_{i,n}^{(l)} \cap \mathcal{I}_{j,n}^{(3-l)} \neq \emptyset\}}.
\end{multline}
since 
\begin{align*}
\mathbb{E}\big[\hat{\eta}_{n,m}^{(3-l)}(t_{i,n}^{(l)}) \big| \mathcal{F} \big]&=n\sum_{k=-K_n}^{K_n} \big| \mathcal{I}^{(l)}_{i+k} \big| \big( \sum_{j=-K_n}^{K_n} \big| \mathcal{I}^{(l)}_{i+j} \big|\big)^{-1}
\mathcal{M}(t^{(3-l)}_{i+k,n})
\\&\leq \sup_{k=-K_n,\ldots,K_n} n\mathcal{M}_n^{(3-l)}(t_{i+k,n}^{(l)})\leq 3 n|\pi_n|_T.
\end{align*}
Because of Theorem \ref{clt} the sum in \eqref{emp_prop2b} is of order $1/n$, while $|\pi_n|_T/b_n \overset{\mathbb{P}}{\longrightarrow} 0$ for $n \rightarrow \infty$ by Condition \ref{condition_test}. Hence, \eqref{emp_prop2b} vanishes.

Next we deal with the increments over non-overlapping observation intervals in the right hand side of \eqref{emp_prop2a}. An upper bound is obtained by taking iterated $\mathcal{S}$-conditional expectations using Lemma \ref{elem_ineq} and the H\"older inequality as in \eqref{Hoelder}, and it is given by
\begin{align*}
&\sum_{l=1,2} \sum_{t_{i,n}^{(l)}\leq T} \big(K_q|\mathcal{I}_{i,n}^{(l)}|^2+K|\mathcal{I}_{i,n}^{(l)}|^{(p'+1-2 \varpi)/p'}+K e_q |\mathcal{I}_{i,n}^{(l)}| \big)    \frac{2K \big(b_n+|\pi_n|_T\big)}{2b_n}
\\&~~~~~~~~~~~~~~~~~~~~~~~~~~~~~~~~~~~~~~~~~~~~ \times  n\sum_{k=-K_n}^{K_n} \big| \mathcal{I}^{(l)}_{i+k} \big| \big( \sum_{j=-K_n}^{K_n} \big| \mathcal{I}^{(l)}_{i+j} \big|\big)^{-1}
\mathcal{M}(t^{(3-l)}_{i+k,n})
\\&~~\leq K \big(K_q|\pi_n|_T+(|\pi_n|_T)^{(1-2 \varpi)/p'}+e_q \big) O_{\mathbb{P}}(1) 
\\&~~~~~~~~~~~~~~~~~~~~~~~~ \times
 n \sum_{l=1,2} \sum_{t_{i,n}^{(l)}\leq T} |\mathcal{I}_{i,n}^{(l)}|
\sum_{k=-K_n}^{K_n} \big| \mathcal{I}^{(l)}_{i+k} \big| \big( \sum_{j=-K_n}^{K_n} \big| \mathcal{I}^{(l)}_{i+j} \big|\big)^{-1}
\mathcal{M}(t^{(3-l)}_{i+k,n}).
\end{align*}
Now \eqref{emp_prop2} follows from Condition \ref{cond_centr}(ii) because of
\begin{align*}
 &n \sum_{l=1,2} \sum_{t_{i,n}^{(l)}\leq T} |\mathcal{I}_{i,n}^{(l)}|
\sum_{k=-K_n}^{K_n} \big| \mathcal{I}^{(l)}_{i+k} \big| \big( \sum_{j=-K_n}^{K_n} \big| \mathcal{I}^{(l)}_{i+j} \big|\big)^{-1}
\mathcal{M}(t^{(3-l)}_{i+k,n})
\\&~~=n \sum_{l=1,2} \sum_{t_{i,n}^{(l)},t_{j,n}^{(3-l)}\leq T} |\mathcal{I}_{i,n}^{(l)}|  |\mathcal{I}_{j,n}^{(3-l)}| \mathds{1}_{\{\mathcal{I}_{i,n}^{(l)} \cap \mathcal{I}_{j,n}^{(3-l)} \neq 0\}} \sum_{k=-K_n}^{K_n} \big| \mathcal{I}_{i+k,n}^{(l)} \big| \big(\sum_{m=-K_n}^{K_n} \big| \mathcal{I}_{i+k+m,n}^{(l)} \big|\big)^{-1}
\\&~~\leq n \sum_{l=1,2} \sum_{t_{i,n}^{(l)},t_{j,n}^{(3-l)}\leq T} |\mathcal{I}_{i,n}^{(l)}|  |\mathcal{I}_{j,n}^{(3-l)}| \mathds{1}_{\{\mathcal{I}_{i,n}^{(l)} \cap \mathcal{I}_{j,n}^{(3-l)} \neq 0\}} 
\\&~~~~~~~~\times \Big(
\sum_{k=-K_n}^{0} \big| \mathcal{I}_{i+k,n}^{(l)} \big| \big(\sum_{m=-K_n}^{0} \big| \mathcal{I}_{i+m,n}^{(l)} \big|\big)^{-1}
+\sum_{k=0}^{K_n} \big| \mathcal{I}_{i+k,n}^{(l)} \big| \big(\sum_{m=0}^{K_n} \big| \mathcal{I}_{i+m,n}^{(l)} \big|\big)^{-1} \Big)
\\&~~\leq 2 H_n(T)
\end{align*}
and $q \to \infty$ afterwards.
\\ \textit{Step 3.} Using dominated convergence, $R(J) \overset{\widetilde{\mathbb{P}}}{\longrightarrow} \widetilde{D}_T$. Also, as the $\mathcal{X}$-conditional distribution of $\widetilde{D}_T$ is continuous on $\Omega_T^{(d)}$, for any choice of $\varepsilon, \eta > 0$ there exists $\delta > 0$ such that
\begin{align*}
\widetilde{\mathbb{P}}\big( \big|\widetilde{\mathbb{P}}\big(\widetilde{D}_T \leq Z \big| \mathcal{X} \big)
-\widetilde{\mathbb{P}}\big(\widetilde{D}_T \pm \delta \leq Z \big| \mathcal{X} \big) \big|>\eta \big) < \varepsilon.
\end{align*}
Then it is easy to deduce that 
\begin{align}\label{emp_prop1a}
\widetilde{\mathbb{P}}\big(R(J) \leq Z \big| \mathcal{X} \big)
\overset{\widetilde{\mathbb{P}}}{\longrightarrow}
\widetilde{\mathbb{P}}\big(\widetilde{D}_T \leq Z \big| \mathcal{X} \big)
\end{align}
holds for $J\rightarrow \infty$. 
\\ \textit{Step 4.} For any $\varepsilon>0$ we have
\begin{align*}
&\widetilde{\mathbb{E}}\big[\big|\frac{1}{M_n} \sum_{m=1}^{M_n} \mathds{1}_{\{R(J,n,m)\leq Z\}}-
\frac{1}{M_n} \sum_{m=1}^{M_n} \mathds{1}_{\{\widehat{D}_{T,n,m}\leq Z\}}\big|\big] \nonumber
\\&~~\leq \widetilde{\mathbb{E}}\big[\frac{1}{M_n} \sum_{m=1}^{M_n} \mathds{1}_{\{|R(J,n,m)-\widehat{D}_{T,n,m} | \geq |R(J,n,m)-Z|\}}\big] \nonumber
\\ &~~\leq \widetilde{\mathbb{E}}\big[\frac{1}{M_n}\sum_{m=1}^{M_n}\big( \mathds{1}_{\{|R(J,n,m)-\widehat{D}_{T,n,m} | > \varepsilon\}}+\mathds{1}_{\{ |R(J,n,m)-Z| \leq \varepsilon\}}\big)\big]. 
\end{align*}
As in \eqref{emp_prop1}, we obtain
\begin{align}\label{emp_prop3b}
\widetilde{\mathbb{E}}\big[\frac{1}{M_n} \sum_{m=1}^{M_n}\mathds{1}_{\{|R(J,n,m)-Z|\leq \varepsilon\}}\big]\rightarrow \widetilde{\mathbb{P}} ( |R(J)-Z|\leq \varepsilon),
\end{align}
and the right hand side tends to zero as $\varepsilon \rightarrow 0$ because the $\mathcal{X}$-conditional distribution of $R(J)$ is continuous, while $Z$ is $\mathcal{X}$-measurable. By \eqref{emp_prop2} we also have
\begin{align}\label{emp_prop3c}
\lim_{J \rightarrow \infty} \limsup_{n \rightarrow \infty}\widetilde{\mathbb{E}}\big[\frac{1}{M_n} \sum_{m=1}^{M_n} \mathds{1}_{\{|R(J,n,m)-\widehat{D}_{T,n,m}| > \varepsilon\}}\big] \rightarrow 0
\end{align}
for all $\varepsilon>0$. Thus, using \eqref{emp_prop3b} and \eqref{emp_prop3c}, we obtain
\begin{align}\label{emp_prop1b}
\lim_{J \rightarrow \infty} \limsup_{n \rightarrow \infty} \widetilde{\mathbb{P}} \big(\big|\frac{1}{M_n} \sum_{m=1}^{M_n} \big(\mathds{1}_{\{R(J,n,m)\leq Z\}}-
 \mathds{1}_{\{\widehat{D}_{T,n,m}\leq Z\}}\big) \big|> \varepsilon \big) =0
\end{align}
for all $\varepsilon>0$.
\\ \textit{Step 5.} The claim follows from \eqref{emp_prop1}, \eqref{emp_prop1a} and \eqref{emp_prop1b}.
\end{proof}

\begin{proof}[Proof of \eqref{estimator_jump}] Again  we will prove the result only for $\widehat{D}_{T,n,m}$. We have for arbitrary $\varepsilon>0$
\begin{align*}
&\widetilde{\mathbb{P}}\big(\widehat{Q}_{n,T}(\alpha)>Q(\alpha)+\varepsilon\big)
\\&~~=\widetilde{\mathbb{P}}\big(\frac{1}{M_n}  \sum_{m=1}^{M_n} \mathds{1}_{\{ \widehat{D}_{T,n,m}> Q(\alpha)+\varepsilon\}}>\frac{M_n-(\lfloor \alpha M_n \rfloor-1)}{M_n} \big) 
\\&~~\leq \widetilde{\mathbb{P}}\big(\frac{1}{M_n} \sum_{m=1}^{M_n} \mathds{1}_{\{ \widehat{D}_{T,n,m}> Q(\alpha)+\varepsilon\}}-Z(\alpha,\varepsilon)> (1- \alpha)-Z(\alpha,\varepsilon) \big)
\end{align*}
with $Z(\alpha,\varepsilon)=\widetilde{\mathbb{P}}\big(\widetilde{D}_T > Q(\alpha)+\varepsilon \big| \mathcal{X} \big)$. Because the $\mathcal{X}$-conditional distribution of $\widetilde{D}_T$ is continuous with a strictly positive density on $[0,\infty)$, we have $Z(\alpha,\varepsilon)<1-\alpha$ a.s. which yields
\begin{align*}
\widetilde{\mathbb{P}}\big(\widehat{Q}_{n,T}(\alpha)>Q(\alpha)+\varepsilon\big) \rightarrow 0
\end{align*}
by \eqref{emp_distr_hatD}. Analogously, we get $\widetilde{\mathbb{P}}\big(\widehat{Q}_{n,T}(\alpha)<Q(\alpha)-\varepsilon\big) \rightarrow 0$. 
\end{proof}

\subsection{Proof of Example \ref{poiss2}}
First, note that $\hat{\eta}_{n,m}^{(l_i)}(s_i)$ and $\hat{\eta}_{n,m}^{(l_j)}(s_j)$ are $\mathcal S$-conditionally independent if we are on the set $\Omega(n,s_i,s_j)$ on which $\hat{\eta}_{n,m}^{(l_i)}(s_i)$ and $\hat{\eta}_{n,m}^{(l_j)}(s_j)$ contain no common observation intervals. Without loss of generality let $s_i < s_j$. Using the Markov inequality we get 
\begin{align}\label{expoisshelp2} 
\mathbb{P}(\Omega(n,s_i,s_j)^c) &\leq \mathbb{P}\big(\tau_{n,+}^{(l_i)}(t_{i_n^{(3-l_i)}(s_i)+K_n,n}^{(3-l_i)})
\geq s_i+(s_j-s_i)/2\big) \nonumber
\\&~~~~~+\mathbb{P}\big(\tau_{n,-}^{(l_j)}(t_{i_n^{(3-l_j)}(s_j)-K_n-1,n}^{(3-l_j)})\leq s_i+(s_j-s_i)/2\big) \nonumber
\\ & \leq 2 K_{\lambda_1,\lambda_2}\frac{K_n/n}{(s_j-s_i)/2}
\end{align}
for a generic constant $K_{\lambda_1,\lambda_2}$. The latter tends to zero as $n \rightarrow \infty$ because of $|\pi_n|_T K_n \overset{\mathbb{P}}{\longrightarrow}0$ and $n |\pi_n|_T = O_{\mathbb{P}}(1)$.
Hence, we may assume $\hat{\eta}_{n,m}^{(l_i)}(s_i)$ and $\hat{\eta}_{n,m}^{(l_j)}(s_j)$ to be $\mathcal S$-conditionally independent, and it remains to prove \eqref{vert_kon1} for $J=1$. Also, we have seen in Example \ref{poiss1} that $\eta^{(l)}(s)$ follows a continuous distribution. If we establish weak convergence of the $\mathcal S$-conditional distribution of $\hat{\eta}^{(l)}_{n,1}(s)$ to the one of ${\eta}^{(l)}(s)$, then (\ref{vert_kon1}) follows from the Portmanteau theorem. 

First, we have
\begin{align*}
&\mathbb{E}\big[\exp(itn\big|\mathcal{I}_{i_n^{(3-l)}(s)+V_{n,1}^{(3-l)}(s),n}^{(3-l)}\big|) \big| \mathcal S \big]
\\&~~~=\sum_{k=-K_n}^{K_n}\big|\mathcal{I}_{i_n^{(3-l)}(s)+k,n}\big| \Big(\sum_{j=-K_n}^{K_n}\big|\mathcal{I}_{i_n^{(3-l)}(s)+j,n}\big|  \Big)^{-1}\exp(itn\big|\mathcal{I}_{i_n^{(3-l)}(s)+k,n}\big|).
\end{align*}
Except for $k=0$ the length of each observation interval is exponentially distributed, up to asymptotically negligible boundary effects, with parameter $n \lambda_{3-l}$. It follows easily that the previous expression has asymptotically the same distribution as
\begin{align*}
 \Big(\sum_{j=-K_n}^{K_n}  E_j  \Big)^{-1} \sum_{j=-K_n}^{K_n}E_j \exp(itE_j)  
\end{align*}
for i.i.d.\ exponentials $E_j$, $j=-K_n,\ldots,K_n$, with parameter $\lambda_{3-l}$. Using the law of large numbers, the $S$-conditional characteristic function of $n|\mathcal{I}_{i_n^{(3-l)}(s)+V_{n,1}^{(3-l)}(s),n}^{(3-l)}|$ converges in probability to 
\begin{align}\label{expoisshelp4}
\mathbb{E}\big[\lambda_{3-l} E_0 \exp(itE_0)  \big]&=\int_0^\infty \lambda_{3-l} x e^{itx} \lambda_{3-l} e^{- \lambda_{3-l} x} dx \nonumber
\\&= \int_0^\infty e^{itx} \frac{(\lambda_{3-l})^2}{\Gamma(2)}xe^{-\lambda_{3-l} x}\nonumber \\&=\mathbb{E}[\exp(itG)]
\end{align}
for a $\Gamma(2,\lambda_{3-l})$-distributed random variable $G$. The latter is the limit distribution of $n|\mathcal{I}_{i_n^{(3-l)}(s)+V_{n,1}^{(3-l)}(s),n}^{(3-l)}|$. 

Finally, as the observation times $t_{i,n}^{(l)}$ as well as the corresponding increments of the Brownian motion $W$ are independent of the observation times $t_{i,n}^{(3-l)}$, and because of the stationarity of increments of the Poisson process, the distribution of the $|\mathcal{I}^{(l)}_{j,n}|$ which overlap with $\mathcal{I}_{i_n^{(3-l)}(s)+V_{n,1}^{(3-l)}(s),n}^{(3-l)}$ depends only on the length of the interval and (asymptotically) not on its location. Therefore, the $\mathcal{I}_{i_n^{(3-l)}(s)+V_{n,1}^{(3-l)}(s),n}^{(3-l)}$-conditional distribution of $\hat{\eta}_{n,1}^{(l)}(s)$ is (asymptotically) completely determined by the length of the latter interval. This reasoning shows that
\begin{align*}
&\mathbb{E}\big[\exp(it \hat{\eta}_{n,1}^{(l)}(s)) \big| \mathcal S \big]
\\&~=\sum_{k=-K_n}^{K_n}\big|\mathcal{I}_{i_n^{(3-l)}(s)+k,n}\big| \big(\sum_{j=-K_n}^{K_n}\big|\mathcal{I}_{i_n^{(3-l)}(s)+j,n}\big|  \big)^{-1}\mathbb{E} \left[\exp\big(itn \eta_n^{(l)}\big(t^{(3-l)}_{i_n^{(3-l)}(s)+k,n}\big)\big)\middle| \mathcal{S} \right]
\end{align*}
has asymptotically the same distribution as
\begin{align}\label{expoisshelp3}
 \Big(\sum_{k=-K_n}^{K_n}E_k  \Big)^{-1}\sum_{k=-K_n}^{K_n}E_k\mathbb{E} \left[\exp\big(it f_{\eta(s)}(E_k,(R_{k,j})_j,(N_{k,j})_{j},P_k)\big)\middle| (E_k,(R_{k,j})_j,P_k)_k \right]
\end{align}
for i.i.d.\ exponentials $E_k$ with parameter $\lambda_{3-l}$, and where we set
\begin{align*}
f_{\eta(s)}(E_k,(R_{k,j})_j,(N_{k,j})_{j},P_k)=\sum_{j=1}^{P_k+1} (R_{k,(j)}-R_{k,(j-1)})(N_{k,j})^2
\end{align*}
with the right hand side defined as in \eqref{poiss1_representation} but with $E_1^{(3-l)}+E_2^{(3-l)}$ replaced by $E_k$. Here, the $E_k$ are independent of all other used random variables. Note that we can choose the $(f_{\eta(s)}(E_k,(R_{k,j})_j,(N_{k,j})_{j},P_k))_k$ to be independent, because the $\eta_n^{(l)}\big(t^{(3-l)}_{i_n^{(3-l)}(s)+k,n}\big)$ are asymptotically independent for large differences of $k$ by \eqref{expoisshelp2}. 
\\\eqref{expoisshelp3} then converges by the law of large numbers to
\begin{multline*}
 \mathbb{E}\left[\lambda_{3-l} E_0 \mathbb{E} \left[\exp\big(it f_{\eta(s)}(E_0,(R_{0,j})_j,(N_{0,j})_{j},P_0)\middle| E_0,(R_{0,j})_j,P_0 \right] \right]
 \\=\mathbb{E}\left[\lambda_{3-l} E_0 \exp\big(it f_{\eta(s)}(E_0,(R_{0,j})_j,(N_{0,j})_{j},P_0)\right],
\end{multline*}
which equals
\begin{align*}
\mathbb{E}\left[\exp\big(it f_{\eta(s)}(G,(R_{0,j})_j,(N_{0,j})_{j},P_0) \right]
\end{align*}
by \eqref{expoisshelp4} for a $\Gamma(2,\lambda_{3-l})$-distributed random variable $G$ because $E_0$ is independent of $(R_{0,j})_j,(N_{0,j})_{j},P_0$. Noting that $E_1^{(3-l)}+E_2^{(3-l)}\sim\Gamma(2,\lambda_{3-l})$ in \eqref{poiss1_representation} finishes the proof. \qed

\bibliography{bibliography}

\begin{thebibliography}{}

\bibitem[\protect\citeauthoryear{A{\"{\i}}t-Sahalia and
  Jacod}{A{\"{\i}}t-Sahalia and Jacod}{2009}]{aitjac2009}
A{\"{\i}}t-Sahalia, Y. and J.~Jacod (2009).
\newblock Testing for jumps in a discretely observed process.
\newblock {\em Ann. Statist.\/}~{\em 37\/}(1), 184--222.

\bibitem[\protect\citeauthoryear{A\"it-Sahalia and Jacod}{A\"it-Sahalia and
  Jacod}{2014}]{AitJac14}
A\"it-Sahalia, Y. and J.~Jacod (2014).
\newblock {\em High-Frequency Finanicial Econometrics}.
\newblock Princeton University Press.
\newblock ISBN: 0-69116-143-3.

\bibitem[\protect\citeauthoryear{Barndorff-Nielsen and
  Shephard}{Barndorff-Nielsen and Shephard}{2006}]{barshe2006}
Barndorff-Nielsen, O. and N.~Shephard (2006).
\newblock Measuring the impact of jumps in multivariate price processes using
  bipower covariation.
\newblock Technical report.

\bibitem[\protect\citeauthoryear{Bibinger and Vetter}{Bibinger and
  Vetter}{2015}]{BibVet15}
Bibinger, M. and M.~Vetter (2015).
\newblock Estimating the quadratic covariation of an asynchronously observed
  semimartingale with jumps.
\newblock {\em Annals of the Institute of Statistical Mathematics\/}~{\em 67},
  707--743.

\bibitem[\protect\citeauthoryear{Bibinger and Winkelmann}{Bibinger and
  Winkelmann}{2015}]{bibwin2015}
Bibinger, M. and L.~Winkelmann (2015).
\newblock Econometrics of co-jumps in high-frequency data with noise.
\newblock {\em J. Econometrics\/}~{\em 184\/}(2), 361--378.

\bibitem[\protect\citeauthoryear{Fukasawa and Rosenbaum}{Fukasawa and
  Rosenbaum}{2012}]{FukRos12}
Fukasawa, M. and M.~Rosenbaum (2012).
\newblock Central limit theorems for realized volatility under hitting times of
  an irregular grid.
\newblock {\em Stoch. proc. appl.\/}~{\em 122\/}(12), 3901--3920.

\bibitem[\protect\citeauthoryear{Hayashi, Jacod, and Yoshida}{Hayashi
  et~al.}{2011}]{hayetal2011}
Hayashi, T., J.~Jacod, and N.~Yoshida (2011).
\newblock Irregular sampling and central limit theorems for power variations:
  the continuous case.
\newblock {\em Ann. Inst. Henri Poincar\'e Probab. Stat.\/}~{\em 47\/}(4),
  1197--1218.

\bibitem[\protect\citeauthoryear{Hayashi and Yoshida}{Hayashi and
  Yoshida}{2005}]{HayYos05}
Hayashi, T. and N.~Yoshida (2005).
\newblock On covariance estimation of non-synchronously observed diffusion
  processes.
\newblock {\em Bernoulli\/}~{\em 11\/}(2), 359--379.

\bibitem[\protect\citeauthoryear{Hayashi and Yoshida}{Hayashi and
  Yoshida}{2008}]{HayYos08}
Hayashi, T. and N.~Yoshida (2008).
\newblock Asymptotic normality of a covariance estimator for non-synchronously
  observed processes.
\newblock {\em Annals of the Institute of Statistical Mathematics\/}~{\em
  60\/}(2), 367--406.

\bibitem[\protect\citeauthoryear{Huang and Tauchen}{Huang and
  Tauchen}{2006}]{HuaTau06}
Huang, X. and G.~Tauchen (2006).
\newblock The relative contribution of jumps to total price variance.
\newblock {\em J. Financial Econometrics\/}~{\em 4}, 456--499.

\bibitem[\protect\citeauthoryear{Jacod}{Jacod}{2008}]{Jac08}
Jacod, J. (2008).
\newblock Asymptotic properties of realized power variations and related
  functionals of semimartingales.
\newblock {\em Stoch. Proc. Appl.\/}~{\em 118\/}(4), 517--559.

\bibitem[\protect\citeauthoryear{Jacod and Protter}{Jacod and
  Protter}{1998}]{JacPro98}
Jacod, J. and P.~Protter (1998).
\newblock Asymptotic error distributions for the euler method for stochastic
  differential equations.
\newblock {\em Ann. Probab.\/}~{\em 26}, 267--307.

\bibitem[\protect\citeauthoryear{Jacod and Protter}{Jacod and
  Protter}{2012}]{JacPro12}
Jacod, J. and P.~Protter (2012).
\newblock {\em Discretization of Processes}.
\newblock Springer.
\newblock ISBN: 3-64224-126-3.

\bibitem[\protect\citeauthoryear{Jacod and Shiryaev}{Jacod and
  Shiryaev}{2002}]{JacShi02}
Jacod, J. and A.~Shiryaev (2002).
\newblock {\em Limit Theorems for Stochastic Processes\/} (2 ed.).
\newblock Springer.
\newblock ISBN: 3-540-43932-3.

\bibitem[\protect\citeauthoryear{Jacod and Todorov}{Jacod and
  Todorov}{2009}]{JacTod09}
Jacod, J. and V.~Todorov (2009).
\newblock Testing for common arrivals of jumps for discretely observed
  multidimensional processes.
\newblock {\em The Annals of Statistics\/}~{\em 37\/}(1), 1792--1838.

\bibitem[\protect\citeauthoryear{Liao and Anderson}{Liao and
  Anderson}{2011}]{liaand2011}
Liao, Y. and H.~Anderson (2011).
\newblock Testing for co-jumps in high-frequency financial data: an approach
  based on first-high-low-last prices.
\newblock Technical report.

\bibitem[\protect\citeauthoryear{Mancini and Gobbi}{Mancini and
  Gobbi}{2012}]{mangob2012}
Mancini, C. and F.~Gobbi (2012).
\newblock Identifying the {B}rownian covariation from the co-jumps given
  discrete observations.
\newblock {\em Econometric Theory\/}~{\em 28\/}(2), 249--273.

\bibitem[\protect\citeauthoryear{Mykland and Zhang}{Mykland and
  Zhang}{2012}]{mykzha2012}
Mykland, P.~A. and L.~Zhang (2012).
\newblock The econometrics of high-frequency data.
\newblock In {\em Statistical methods for stochastic differential equations},
  Volume 124 of {\em Monogr. Statist. Appl. Probab.}, pp.\  109--190. CRC
  Press, Boca Raton, FL.

\bibitem[\protect\citeauthoryear{Podolskij and Vetter}{Podolskij and
  Vetter}{2010}]{PodVet10}
Podolskij, M. and M.~Vetter (2010).
\newblock Understanding limit theorems for semimartingales: a short survey.
\newblock {\em Statistica Neerlandica\/}~{\em 64}, 329--351.

\bibitem[\protect\citeauthoryear{Vetter and Zwingmann}{Vetter and
  Zwingmann}{2016}]{VetZwi16}
Vetter, M. and T.~Zwingmann (2016).
\newblock A note on central limit theorems for quadratic variation in case of
  endogenous observation times.
\newblock {\em preprint, arxiv: 1605.07056\/}.
\newblock \url{http://arxiv.org/abs/1605.07056}.

\end{thebibliography}
\end{document}